\documentclass[11pt]{article}

\usepackage{tikz}
\usetikzlibrary{decorations.pathreplacing,calc,arrows.meta}

\usepackage[a4paper,margin=3cm]{geometry}
\usepackage{amsmath,amssymb,amsthm,mathtools}
\usepackage{bm}
\usepackage{microtype}
\usepackage{enumitem}
\usepackage{hyperref,listings}
\usepackage[nameinlink,capitalise]{cleveref}

\hypersetup{
  colorlinks=true,
  linkcolor=blue,
  citecolor=blue,
  urlcolor=blue
}

\lstdefinelanguage{Maple}{
  morekeywords={
    if, then, else, elif, fi, do, od, for, from, to, by, while, 
    in, and, or, not, proc, end, local, global, return, cat, 
    subs, eval, diff, int, limit, series, map, zip, seq, add
  },
  sensitive=true,
  morecomment=[l]{\#},
  morestring=[b]",
  morestring=[b]',
  basicstyle=\ttfamily\small,
  keywordstyle=\color{blue}\bfseries,
  commentstyle=\color{black}\itshape,
  stringstyle=\color{red},
  showstringspaces=false,
  breaklines=true,
  frame=single
}

\numberwithin{equation}{section}

\newtheorem{theorem}{Theorem}[section]
\newtheorem{proposition}[theorem]{Proposition}
\newtheorem{lemma}[theorem]{Lemma}
\newtheorem{corollary}[theorem]{Corollary}

\theoremstyle{definition}
\newtheorem{definition}[theorem]{Definition}
\newtheorem{assumption}[theorem]{Assumption}

\theoremstyle{remark}
\newtheorem{remark}[theorem]{Remark}
\newtheorem{example}[theorem]{Example}

\newcommand{\R}{\mathbb{R}}
\newcommand{\T}{\mathbb{T}}
\newcommand{\Id}{\mathrm{Id}}
\newcommand{\tr}{\operatorname{tr}}
\newcommand{\diver}{\operatorname{div}}
\newcommand{\A}{\mathcal{A}}
\newcommand{\Lagr}{\mathcal{L}}
\newcommand{\Energy}{\mathcal{E}}
\newcommand{\eint}{e}
\newcommand{\Dcoef}{\mathfrak{D}}
\newcommand{\Ecoef}{\mathfrak{E}}
\newcommand{\Kcoef}{\mathfrak{K}}
\newcommand{\Icoef}{\mathfrak{I}}

\title{  Admissibility criteria for convex integration fan solutions  and contact discontinuities in the Euler equation}

\author{Heiko Gimperlein\thanks{Engineering Mathematics, University of Innsbruck, Innsbruck, Austria.\\ email: heiko.gimperlein@uibk.ac.at}}

\date{}

\begin{document}

\maketitle

\begin{abstract}
\noindent For piecewise constant fan subsolutions to the isentropic Euler equations, the entropy and action rates of any associated convex integration solution,
relative to a common reference solution, depend only on the fan and are computed explicitly. Moreover, these explicit expressions may be decomposed into a kinetic energy mismatch and an internal energy mismatch, from which we characterize agreement of Dafermos' entropy rate criterion with action rate criteria. The known disagreement for the Chiodaroli--Kreml two-shock data is recovered. For the planar contact reference, we prove that both criteria strictly prefer every convex integration solution associated with a strictly dissipative admissible fan to the classical contact discontinuity. Certified computations locate the Krupa--Sz\'{e}kelyhidi solutions in this region, where the criteria agree, while Horimoto's analytical construction provides such fans for every strictly increasing pressure law.
%For the contact-discontinuity solutions of Krupa and Sz\'{e}kelyhidi, certified exact-arithmetic computations show that both criteria prefer the convex integration solutions to the classical contact discontinuity. The same agreement holds, by a purely analytic argument, for Horimoto's contact-discontinuity solutions with arbitrary strictly increasing pressure laws.
\end{abstract}

\section{Introduction}
\label{sec:introduction}

The compressible Euler equations in two or more spatial dimensions admit nonunique bounded entropic weak solutions. In particular, for Riemann initial data, the well-developed theory of unique self-similar entropy solutions in one spatial dimension is in contrast with the infinitely many admissible weak solutions in higher dimensions, constructed by convex integration \cite{ChiodaroliDeLellisKreml, ChiodaroliKreml, DeLellisSzekelyhidi, KrupaSzekelyhidi,m21}.

This raises the basic question whether an additional selection criterion can be formulated that distinguishes  a unique solution. Dafermos' entropy rate criterion \cite{DafermosEntropyRate} has long served as one such candidate. It prefers solutions whose mechanical energy decreases at the greatest instantaneous rate. Related rate criteria based on the action were introduced and studied  \cite{GKKS1, GKKS2}, and both
families of criteria have been tested against convex integration solutions \cite{CFKM,ChiodaroliKreml,Dafermos26,Feireisl,GK,MP1,MP2}. In particular, Feireisl \cite{Feireisl} showed that solutions obtained by convex integration are not entropy rate admissible. Surprisingly, Chiodaroli and Kreml \cite{ChiodaroliKreml} then found Riemann data for which the entropy rate criterion prefers convex integration solutions to the
classical two-shock solution. For the same data the action rate criteria prefer the classical two-shock solution \cite{GKKS1,GKKS2}, whereas for other data, action criteria prefer suitably constructed convex integration solutions \cite{MP1, MP2}. Note that these are statements of a different kind than the non-admissibility result of \cite{Feireisl}: they compare convex integration solutions pairwise with one designated reference solution, rather than within a whole class. How the two pairwise criteria relate to each other, however, has remained unclear. 

The purpose of the present paper is to make this relationship explicit and readily computable. Action and entropy rate criteria involve the
same kinetic energy density and compare solutions with the same initial data, but differ in the sign of the thermodynamic contribution. For the convex integration solutions to the Euler  equations generated by an arbitrary piecewise constant fan subsolution, we show that both comparisons with a piecewise constant self-similar reference solution reduce to explicit sums over the fan parameters. These expressions decompose into a kinetic energy mismatch $\Kcoef$ and an internal energy mismatch $\Icoef$. The two criteria agree if and only if $\Kcoef$ dominates $\Icoef$: $$|\Icoef| < \frac{|\Kcoef|}{2}.$$
Figure \ref{fig:KI-phase-diagram} displays the resulting selection regions\footnote{In Figure \ref{fig:KI-phase-diagram} points are drawn in the correct direction from the origin, so that sector membership and the ratios $\Icoef/\Kcoef$ are exact, but radial distances are rescaled for legibility.}. There, the two-shock family of Chiodaroli and Kreml traces a curve which starts on the negative $\Kcoef$-axis, where both criteria prefer the reference solution. It crosses into the lower sector, where the entropy criterion prefers the convex integration solutions \cite[Theorem 2]{ChiodaroliKreml}, while action rate criteria continue to prefer the reference solution \cite[Theorem 6]{GKKS1}. 

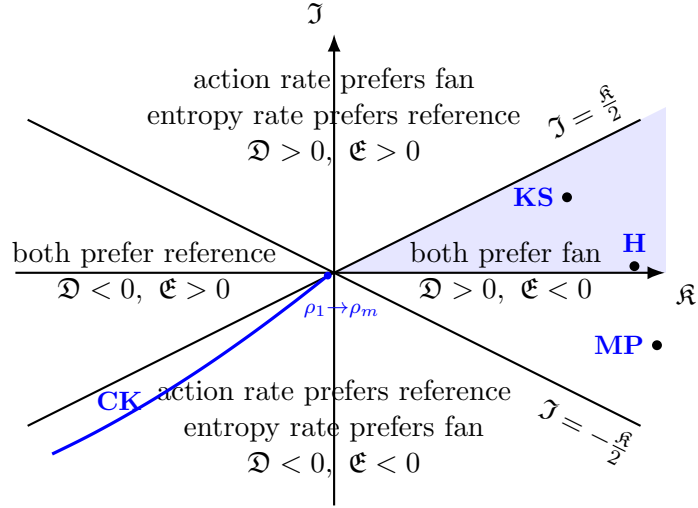
\begin{figure}[htbp]
\centering
\begin{tikzpicture}[scale=0.81,>=Latex]

\begin{scope}
  \clip (-5.2,-3.8) rectangle (5.4,3.9);
  \fill[blue!10] (0,0) -- (100,0) -- (100,50) -- cycle;
\end{scope}

% Axes
\draw[->,thick] (-5.2,0) -- (5.4,0)
node[below right] {$\Kcoef$};
\draw[->,thick] (0,-3.8) -- (0,3.9)
node[above left] {$\Icoef$};

% Boundary lines I = +/- K/2
\draw[thick] (-5,-2.5) -- (5,2.5)
node[pos=0.93,above,sloped]
{$\Icoef=\frac{\Kcoef}{2}$};

\draw[thick] (-5,2.5) -- (5,-2.5)
node[pos=0.93,below,sloped]
{$\Icoef=-\frac{\Kcoef}{2}$};

% Region labels
\node[align=center] at (2.8,0)
{{both prefer fan}\\
$\Dcoef>0,\ \Ecoef<0$};

\node[align=center] at (-3.1,0)
{{both prefer reference}\\
$\Dcoef<0,\ \Ecoef>0$};

\node[align=center] at (0,2.55)
{{action rate prefers fan}\\
{entropy rate prefers reference}\\
$\Dcoef>0,\ \Ecoef>0$};

\node[align=center] at (0,-2.55)
{{action rate prefers reference}\\
{entropy rate prefers fan}\\
$\Dcoef<0,\ \Ecoef<0$};

\coordinate (KS) at (3.8,1.2331);
\fill (KS) circle (2.3pt);
\node[left=1pt,font=\small] at (KS) {\textcolor{blue}{\textbf{KS}}};

\draw[very thick,blue] plot[smooth] coordinates {(-4.600,-2.960)
(-4.314,-2.821) (-4.040,-2.683) (-3.778,-2.544) (-3.526,-2.405)
(-3.283,-2.266) (-3.047,-2.128) (-2.819,-1.989) (-2.598,-1.850)
(-2.382,-1.711) (-2.172,-1.573) (-1.966,-1.434) (-1.765,-1.295)
(-1.568,-1.156) (-1.375,-1.018) (-1.185,-0.879) (-0.998,-0.740)
(-0.815,-0.601) (-0.633,-0.463) (-0.455,-0.324) (-0.278,-0.185)
(-0.104,-0.046)};
\fill[blue] (-0.104,-0.046) circle (1.8pt);
\node[blue,font=\small] at (-3.5,-2.1) {\textbf{CK}};
\node[blue,font=\scriptsize,below right=-6pt] at (-0.4,-0.60)
{$\rho_1{\to}\rho_m$};

\fill[black] (4.9,0.108) circle (2.3pt);
\node[blue,font=\small,above=1.5pt] at (4.9,0.108) {\textbf{H}};

\coordinate (MP) at (5.268, -1.183);
\fill (MP) circle (2.3pt);
\node[left=1pt,font=\small] at (MP) {\textcolor{blue}{\textbf{MP}}};

\end{tikzpicture}
\caption{
Selection regions in the
$(\Kcoef,\Icoef)$-plane. Action and entropy criteria agree in the left and right sectors, with solutions for the classical contact discontinuity in the shaded region (Theorem \ref{thm:contact-agreement}). Point KS shows the location of the Krupa-Sz\'ekelyhidi solution \cite{KrupaSzekelyhidi} (Sections \ref{sec:KS-application}, \ref{sec:KS-stability}), H marks an example of a relaxed Horimoto fan \cite{Horimoto2026} (Section \ref{sec:Horimoto}). The solution in  point MP \cite{MP2} includes both shock and contact (Proposition \ref{prop:MP-point}). Curve CK shows a family of two-shock solutions from \cite{ChiodaroliKreml} crossing two selection regions (Example \ref{exampleKS}).}
\label{fig:KI-phase-diagram}
\end{figure}

The main application concerns the planar contact discontinuity. This solution is unique in one dimension and produces no entropy on the contact interface. For strictly dissipative fan solutions we find that both entropy and action rate criteria prefer the convex integration solutions to the one-dimensional self-similar solution and, specifically, that these solutions lie in the shaded sector in Figure \ref{fig:KI-phase-diagram}. This explains the common outcome of the Krupa-Sz\'ekelyhidi and Horimoto constructions \cite{KrupaSzekelyhidi,Horimoto2026}. The numerical Krupa--Sz\'ekelyhidi construction and its certification quantify the location of one particular fan, while Horimoto's construction shows
that the same conclusion holds analytically for arbitrary strictly increasing pressure laws.
Convex integration solutions from \cite{MP2} include both shock and contact and lie outside the shaded sector, but are still preferred by both entropy and action rate criteria. %If the contact discontinuity were to describe the correct physics in $2$d, a new criterion would  be required. 
Correspondingly,  Krupa and Sz\'{e}kelyhidi \cite{KrupaSzekelyhidi} suggest that the non-uniqueness of ``wild'' solutions reflects the strong instability of the one-dimensional Riemann solution with respect to genuine
multidimensional perturbations. It may also be indicative of a compressible Kelvin-Helmholtz instability and the onset of turbulence. The results of this article are consistent with this picture: rather than ``failing'', the selection criteria may correctly reflect the physics by preferring dissipating, two-dimensional solutions to the non-dissipative planar contact.\\

The main contributions of this paper are the following.\\

\vspace*{-0.3cm}

\noindent 1)  
For arbitrary piecewise constant fan subsolutions, we derive explicit finite-dimensional formulas for the differences between action functionals and between entropy rate functionals. Two coefficients determine the comparison: the energy rate difference $\Ecoef$ between the fan and reference solutions, and the corresponding action rate difference $\Dcoef$. The entropy rate criterion strictly prefers the fan solutions if $ \Ecoef<0$, while the action rate criterion strictly prefers the fan solutions if $\Dcoef>0$. The two coefficients admit the decomposition $\Dcoef=\frac12\Kcoef+\Icoef$, $\Ecoef=-\frac12\Kcoef+\Icoef,$ into a weighted kinetic energy mismatch $\Kcoef$  and a weighted internal energy mismatch $\Icoef$.  The criteria agree if and only if $|\Icoef|<|\Kcoef|/2$. For the classical planar contact, Theorem \ref{thm:contact-agreement} shows that strictly dissipative fan solutions satisfy $0\leq \Icoef < \Kcoef/2$. In particular, such solutions lie in the shaded right hand sector in Figure \ref{fig:KI-phase-diagram}, and both admissibility criteria prefer these convex integration solutions to the classical contact discontinuity. %Examples include the Krupa-Sz\'ekelyhidi and Horimoto solutions. %Their disagreement in the Chiodaroli-Kreml two-shock family (Example \ref{exampleKS}), in the lower sector in Figure \ref{fig:KI-phase-diagram}, recovers the results of \cite{ChiodaroliKreml} and \cite{GKKS1}.
\\

\vspace*{-0.3cm}

\noindent 2) The general theorem is applied quantitatively to the numerically constructed solutions for the contact discontinuity from \cite{KrupaSzekelyhidi}, as well as to the analytical solutions from  \cite{Horimoto2026} for a general class of pressure laws. Ancillary codes rigorously verify the numerical results for the solutions from  \cite{KrupaSzekelyhidi} using exact rational arithmetic with integer-square-root enclosures and $LDL^{\top}$ certificates. Both stability and verification are formulated abstractly, irrespective of the existence of a solution. In this way the results allow one to certify conclusions for any given solution. \\

\vspace*{-0.1cm}

Note that the criteria studied here apply to the class of bounded entropic weak solutions. A complementary viewpoint suggests that the wild non-uniqueness instead calls for a revised concept of solutions, such as measure-valued or dissipative solutions \cite{BFH,GSGW}. The present results are compatible with this view and quantify it: by Proposition \ref{prop:rigidity-fixed-fan} the comparisons depend only on the fan subsolution, i.e.~the averaged, measure-valued description of the oscillatory solutions, and the outcome at the contact discontinuity clarifies the scope of what rate-based selection can achieve within this class.\\

\vspace*{-0.2cm}

\noindent \emph{Outline of this article:} Section \ref{sec:equations-fans} recalls the Euler equations, admissibility criteria,
and the properties of convex integration solutions associated to fan subsolutions. Section \ref{sec:finite-dimensional-formulas}  derives explicit formulas for the action and entropy rate in terms of the finite number of fan and reference parameters. Section \ref{sec:comparisonstability} establishes the decomposition of the rate coefficients, characterizes the regions of agreement and disagreement between the two criteria, and shows their stability under parameter perturbations. The decomposition is illustrated for the two-shock solutions from \cite{ChiodaroliKreml}. The general theory is then applied to the Krupa-Sz\'ekelyhidi fan solutions in Sections \ref{sec:KS-application} and \ref{sec:KS-stability}. These sections also discuss the passage from approximate numerical solutions to exact solutions. The symbolic verification and certification of the numerical results is the content of Appendix \ref{sec:symbolic-verification}. Section \ref{sec:Horimoto} discusses Horimoto's solutions for the contact discontinuity.\\

\vspace*{-0.2cm}

\noindent \emph{Acknowledgements:} The author thanks Marshall Slemrod for pointing out this problem and Michael Grinfeld, Robin J.~Knops, Valentin Pellhammer and Marshall Slemrod for feedback and discussions.

\section{The compressible Euler equations: admissibility criteria and fan subsolutions}
\label{sec:equations-fans}

\textbf{The isentropic Euler equations.} We consider the compressible barotropic Euler equations for the scalar density $\rho(x,t)>0$ and vector velocity $v$ at time $t \in (0,T)$ and  $x\in \Omega \subset \mathbb{R}^2$ given by
\begin{align}  \partial_t\rho+\diver_x(\rho v)&=0,  \label{eq:continuity}  \\
  \partial_t(\rho v)  +  \diver_x(\rho v\otimes v) +  \nabla_xp(\rho)&=0, \label{eq:momentum}
\end{align}
with prescribed initial conditions
\begin{equation}
\label{baric} (\rho,\,v)(0,\cdot)=(\rho_{0},v_{0}). 
\end{equation}
The constitutively defined pressure $p=p(\rho)$, assumed to be sufficiently smooth and monotonically increasing ($p'(\rho)>0$), is related to the specific internal energy $\eint = \eint(\rho)$ by $p(\rho)=\rho^2\eint'(\rho)$.
The energy density is then given by the pointwise sum of kinetic and potential energy:
\begin{equation}
  H(\rho,v)  :=  \frac12\rho|v|^2+\rho\eint(\rho). \label{eq:mechanical energy}
\end{equation}
Smooth solutions of the Euler equations \eqref{eq:continuity}, \eqref{eq:momentum} satisfy a local energy-entropy identity
\begin{equation}  \label{eq:energy-equality}
  \partial_tH(\rho,v)  +  \diver_x  \left[    \bigl(H(\rho,v)+p(\rho)\bigr)v  \right]  =0.
\end{equation}
\begin{remark}
Note that the relation $p(\rho)=\rho^2\eint'(\rho)$ only determines $\eint$ up to an additive constant. 
It is readily verified that \eqref{eq:energy-equality} and later conclusions in this paper are independent of this choice.
\end{remark}
We consider \emph{entropic} weak solutions to the  initial value problem \eqref{eq:continuity}-\eqref{baric}, that is, weak solutions that satisfy an energy-entropy condition in the sense of distributions (see, for example, \cite{ChiodaroliKreml,DafermosEntropyRate,m21}).
\begin{definition}Recall that $\Omega \subset \R^2$.\\
a) The functions $(\rho,v)  \in  L^\infty(\Omega\times(0,T))  \times  L^\infty(\Omega\times(0,T);\R^2)$ with $\rho>0$ a.~e.~define a weak solution of \eqref{eq:continuity}-\eqref{baric}, if the equations hold in the sense of
distributions.

\noindent b) A weak solution $(\rho,v)$ is said to be entropic if it satisfies
\begin{equation}\label{eq:energy-inequality}   \partial_tH(\rho,v)  + \diver_x
  \left[    \bigl(H(\rho,v)+p(\rho)\bigr)v  \right]  \leq 0  
\end{equation}
in the sense of distributions.
\end{definition}
Inequality \eqref{eq:energy-inequality} describes the non-increase of the total mechanical energy.
In more than one spatial dimension, however, it does not imply uniqueness of solutions. Consequently, it is desirable to study admissibility criteria that select a particular solution. \\

\noindent \textbf{Admissibility criteria.} In a domain $\Omega \subset \mathbb{R}^2$ Dafermos' entropy rate criterion \cite{DafermosEntropyRate} compares the instantaneous right derivatives of the total entropy, defined by:
\begin{equation}  \label{eq:total-energy}
  \Energy(t)  := \Energy(\rho,v,t)  :=  \int_\Omega H(\rho,v)(x,t)\,dx,
\end{equation}
provided the integral is finite. Since the physical entropy for the
barotropic Euler system is the mechanical energy, the preferred
solution is the one with the smallest energy rate.

We formulate this criterion as a pairwise comparison.
\begin{definition}
Let $S$ be a set of entropic weak solutions of the initial value problem \eqref{eq:continuity}-\eqref{baric}, for given initial data.\\
$(i)$ A solution $(\rho^{(1)},v^{(1)}) \in S$ is preferred to $(\rho^{(2)},v^{(2)})\in S$ by the 
entropy rate  criterion at time $t$, when
$$  \frac{d}{dt}\Big|_{t^+}\left(  \Energy(\rho^{(1)},v^{(1)},t)- \Energy(\rho^{(2)},v^{(2)},t)\right) \leq 0.$$
When strict inequality holds, then $(\rho^{(1)},v^{(1)})$ is strictly preferred to $(\rho^{(2)},v^{(2)})$.\\
$(ii)$ A solution $(\rho^{(1)},v^{(1)}) \in S$ is (strictly) entropy rate admissible in $S$ at time $t$, when for every  $(\rho^{(2)},v^{(2)}) \in S$ with $(\rho^{(2)},v^{(2)}) \neq (\rho^{(1)},v^{(1)})$, it is (strictly) preferred to $(\rho^{(2)},v^{(2)})$ by the entropy rate  criterion at time $t$. 
\end{definition}
Often, $t$ is taken to be the initial time $t=0$. Since the results of this article are independent of $t$, dependence upon this variable is omitted. \\

Note that a unique entropic weak solution is entropy rate admissible as the comparison set $S$ becomes a singleton. In a class $S$ of all entropic weak solutions to \eqref{eq:continuity}-\eqref{baric} constructed by convex
integration, under some mild restrictions Feireisl showed in \cite{Feireisl} that there is no entropy rate admissible solution. Surprisingly, Chiodaroli and Kreml \cite{ChiodaroliKreml} then found a two-dimensional Riemann problem in which
convex integration solutions could be preferred, according to the
entropy rate criterion, to the perhaps physically expected two shock
solution. Their results motivated the introduction of alternative admissibility criteria \cite{GKKS1, GKKS2}, which were shown to prefer the two shock solution for certain Riemann data. These criteria are based on the action
\begin{equation}
  \A(t):= \A(\rho,v,t) :=  \int_{0}^t \int_\Omega  \Lagr(\rho,v)\,dx\,dt,  \label{eq:action}
\end{equation}
instead of energy. Here, the Lagrangian $\Lagr$ denotes the pointwise difference between kinetic and potential energy, 
\begin{equation}\label{eq:lagrangian-density}
  \Lagr(\rho,v)  :=  \frac12\rho|v|^2-\rho\eint(\rho).
\end{equation}

\begin{definition}[\cite{GKKS2}]
Let $S$ be a set of entropic weak solutions of the initial value problem \eqref{eq:continuity}-\eqref{baric}, for given initial data.\\
$(i)$ A solution $(\rho^{(1)},v^{(1)}) \in S$ is preferred to $(\rho^{(2)},v^{(2)})\in S$ under the action rate criterion LAAP${}_0$,  when there exists a time
  $t_{1}=t_{1}(\rho^{(2)},v^{(2)})>0$ such that
  $$\A(\rho^{(1)},v^{(1)},t) - \A(\rho^{(2)},v^{(2)},t) \le 0\qquad \text{for all }
  t\in (0,t_{1}).$$
When strict inequality holds for some $t\in (0,\,t_{1})$, then $(\rho^{(1)},v^{(1)})$ is strictly preferred to $(\rho^{(2)},v^{(2)})$.\\
$(ii)$ A solution $(\rho^{(1)},v^{(1)}) \in S$ is (strictly) action rate, or LAAP${}_0$, admissible in $S$, when for every  $(\rho^{(2)},v^{(2)}) \in S$ with $(\rho^{(2)},v^{(2)}) \neq (\rho^{(1)},v^{(1)})$, it is (strictly) preferred to $(\rho^{(2)},v^{(2)})$ by the 
action rate criterion. 
\end{definition}

In unbounded spatial domains, such as $\Omega = \mathbb{R}^2$, individual actions and total energies may be infinite. In this case energies and actions are compared for solutions which agree outside a bounded set, see Remark \ref{rem:renormal}. \\

\noindent \textbf{Riemann data and fan subsolutions.} It is convenient to record several definitions and known results required subsequently. Thus, Riemann initial data is specified by:
\begin{equation}
  (\rho_0,v_0)  =
  \begin{cases}
(\rho_{-},\,{{v}_{-}}), & x_{2}<0,\ -\infty<x_{1}<\infty, \\
(\rho_{+},\,{{v}_{+}}),& x_{2}>0,\  -\infty <x_{1}<\infty,
  \end{cases}
  \label{eq:riemann-data}
\end{equation}
with a discontinuity at the line $x_2=0$. Here, $\rho_\pm, v_\pm$ are constants. 

The families of non-unique entropic weak solutions constructed by convex integration \cite{ChiodaroliKreml, KrupaSzekelyhidi, m21} replace the trace free part of
$$  v\otimes v =  \left(    v\otimes v-\frac{|v|^2}{2}\Id \right)+\frac{|v|^2}{2}\Id$$
by an independent field $$
  u\in\mathcal S_0^{2\times2} :=  \left\{    u\in\R^{2\times2}:    u^\top=u,\ \tr u=0  \right\}.$$

\begin{definition}
A triple $(v,u,C)  \in  \R^2\times\mathcal S_0^{2\times2}\times(0,\infty)$ satisfies the strict subsolution condition, if
\begin{equation}
  \frac{C}{2}\Id - v\otimes v+u  \label{eq:strict-subsolution-condition}
\end{equation}
is  positive definite.
\end{definition}
On taking the trace in \eqref{eq:strict-subsolution-condition}, we conclude that
$|v|^2<C$, and therefore $ \frac12\rho(C-|v|^2)$
may be interpreted as additional kinetic energy from unresolved oscillations.

A piecewise constant fan subsolution to the Riemann problem is defined next. The fan ansatz partitions spacetime into self-similar conical regions $P_-,P_1,\ldots,P_N,P_+$, separated by the lines $x_2=\nu_i t$, on each of which the coarse density, velocity, and relaxed stress are constant. The intermediate regions $P_i$ are the zones where convex-integration oscillations can be inserted.
 
\begin{definition}[\cite{ChiodaroliKreml,KrupaSzekelyhidi}]\label{def:fan}
An admissible piecewise constant fan subsolution $  (\bar\rho,\bar v,\bar u)$ associated with the Riemann data
\eqref{eq:riemann-data} consists of constants
$$  \rho_i>0,  \qquad  v_i\in\R^2, \qquad  u_i\in\mathcal S_0^{2\times2},  \qquad C_i>0, \qquad (i=1,\ldots,N)$$
and fan speeds
$\nu_0<\nu_1<\cdots<\nu_N$,
such that the piecewise constant density
\begin{equation*}
  \bar\rho  =  \rho_-\mathbf 1_{P_-}  +  \sum_{i=1}^N  \rho_i\mathbf 1_{P_i} +  \rho_+\mathbf 1_{P_+},
\end{equation*}
the piecewise constant velocity
\begin{equation*}
  \bar v  =  v_-\mathbf 1_{P_-}  +  \sum_{i=1}^N  v_i\mathbf 1_{P_i}  +  v_+\mathbf 1_{P_+},
\end{equation*}
and the corresponding relaxed stress satisfy the relaxed Euler
equations, the Rankine--Hugoniot relations, and the admissibility
inequalities at every interface, with
\begin{equation}
  v_i\otimes v_i-u_i  < \frac{C_i}{2}\Id  \qquad  (i=1,\ldots,N).
  \label{eq:fan-strict-subsolution}
\end{equation}
Here, we denote by
\begin{align*}
  P_-  :=  \{(x,t):t>0,\ x_2<\nu_0t\},  P_+  :=  \{(x,t):t>0,\ x_2>\nu_Nt\},
\end{align*}
the  exterior regions and set
\begin{equation*}
  P_i  :=  \{(x,t):t>0,\ \nu_{i-1}t<x_2<\nu_it\},  \qquad  i=1,\ldots,N.
\end{equation*}
\end{definition}
The detailed algebraic form of the relaxed equations here depends on the
chosen convex integration formulation. \\

Convex integration shows that for every admissible piecewise constant fan subsolution $  (\bar\rho,\bar v,\bar u)$  %satisfying \eqref{eq:fan-strict-subsolution}. 
there exist infinitely many bounded entropic weak solutions
$  (\rho,v)$ of the isentropic Euler equations with the same Riemann data such that
$\rho=\bar\rho$ almost everywhere. The solutions agree with the Riemann data in $P_-\cup P_+$,
and $|v|^2=C_i$ almost everywhere in $P_i$, $i=1,\ldots,N$.\\

It is assumed for the admissibility comparisons stated below, that for all $i=1,\ldots,N$:
\begin{enumerate}[label=\textup{(\roman*)}]
  \item $ \rho|_{P_i}$ is constant and $( \rho,  v)|_{P_\pm}$ agrees with the prescribed Riemann data,
  \item convex integration produces bounded entropic weak solutions, and
  \item $  |v|^2=C_i$ almost everywhere in $P_i$.
\end{enumerate}

For later use, define in the respective regions $P_i$ the fan Lagrangian and mechanical energy densities to be
\begin{align}
  \Lambda_i  &:=  \frac12\rho_iC_i-\rho_i\eint(\rho_i), \ \   H_i
  :=  \frac12\rho_iC_i+\rho_i\eint(\rho_i) \qquad (i=1,\ldots,N), \label{eq:fanprof1}\\
  \Lambda_\pm  &:=  \frac12\rho_\pm|v_\pm|^2-\rho_\pm\eint(\rho_\pm),\ \ 
  H_\pm  :=  \frac12\rho_\pm|v_\pm|^2+\rho_\pm\eint(\rho_\pm). \label{eq:fanprof2}
\end{align}

Every convex integration solution associated with the same fan subsolution therefore has the same piecewise constant Lagrangian and mechanical energy profiles, $\Lambda_i, \Lambda_\pm$, respectively $H_i, H_\pm$, in similarity coordinates.

\begin{remark}\label{rem:renormal} As remarked above, in the unbounded spatial domain $\Omega = \mathbb{R}^2$ energies and actions may be infinite.
For solutions associated with the same Riemann data, the differences between the energy and action may still be defined by their values on any sufficiently large, truncated domain. For explicit calculations, we therefore work on a periodic strip
$  \Omega_\ell := \T_\ell\times\R$, where
$\T_\ell:=\R/(\ell\mathbb Z)$,
$\ell>0$ sufficiently large. The energy and action differences are then given by \begin{align}
\Energy(\rho^{(1)},v^{(1)},t)-  \Energy(\rho^{(2)},v^{(2)},t) 
& =\int_{\Omega_\ell}  \left[    H(\rho^{(1)},v^{(1)}) - H(\rho^{(2)},v^{(2)})\right]dx, \label{eq:renormalized-energy-difference}\\
  \A(\rho^{(1)},v^{(1)},t)  - \A(\rho^{(2)},v^{(2)},t) &=  \int_0^t\int_{\Omega_\ell} \left[  \Lagr(\rho^{(1)},v^{(1)}) -\Lagr(\rho^{(2)},v^{(2)}) \right] dx\,dt.  \label{eq:renormalized-action-difference}
\end{align}
\end{remark}

\begin{proposition}
\label{prop:rigidity-fixed-fan}
Let $(\rho,v)$ and  $(\widetilde\rho,\widetilde v)$ be two convex integration realizations generated by the same
admissible fan subsolution. Then
\begin{equation*}
  \Lagr(\rho,v)  =   \Lagr(\widetilde\rho,\widetilde v), \ \   H(\rho,v) =  H(\widetilde\rho,\widetilde v)  \qquad\text{almost everywhere}.
\end{equation*}
Consequently, $(\rho,v)$ and  $(\widetilde\rho,\widetilde v)$ have the same action on every bounded
spacetime set and the same  total-energy rate relative to
any common reference solution.
\end{proposition}
\begin{proof}
Note that $  \rho=\widetilde\rho=\bar\rho$ and that in every region $P_i$,
$ |v|^2=|\widetilde v|^2=C_i$ almost everywhere.
Therefore $ \Lagr(\rho,v) = \frac12\rho_iC_i-\rho_i\eint(\rho_i) =  \Lagr(\widetilde\rho,\widetilde v)$
and $ H(\rho,v) =  \frac12\rho_iC_i+\rho_i\eint(\rho_i) = H(\widetilde\rho,\widetilde v)$
almost everywhere in $P_i$. In  exterior regions, both solutions agree with the prescribed Riemann states.
\end{proof}

\section{Action and entropy rate formulas} \label{sec:finite-dimensional-formulas}

This section computes the action and energy differences between a convex integration fan solution and an arbitrary
piecewise constant self-similar reference solution. The differences reduce to simple expressions in the similarity variable
$\xi=\frac{x_2}{t}$. Note, for example, that with respect to this coordinate the fan region $P_i$ corresponds to $ \nu_{i-1}<\xi<\nu_i$.

A piecewise constant self-similar reference solution $(\rho^{\mathrm{ref}},v^{\mathrm{ref}})$ will be defined in terms of a partition $ -\infty=\sigma_0<\sigma_1<\cdots<\sigma_M=\infty$
of the $\xi$-axis. In an obvious notation, it takes the form
\begin{equation*}
  (\rho^{\mathrm{ref}},v^{\mathrm{ref}})(x,t) =  (\rho_j^{\mathrm{ref}},v_j^{\mathrm{ref}})
  \qquad  \text{whenever }  \sigma_{j-1} < \xi =  \frac{x_2}{t} <  \sigma_j \qquad (j=1,\ldots,M).
\end{equation*}
Furthermore, the mechanical energy and Lagrangian densities, given by  
\begin{equation}
  H_j^{\mathrm{ref}} :=  \frac12  \rho_j^{\mathrm{ref}}  |v_j^{\mathrm{ref}}|^2 +  \rho_j^{\mathrm{ref}} \eint(\rho_j^{\mathrm{ref}}), \quad
  \Lambda_j^{\mathrm{ref}}  :=  \frac12   \rho_j^{\mathrm{ref}}  |v_j^{\mathrm{ref}}|^2 -  \rho_j^{\mathrm{ref}}  \eint(\rho_j^{\mathrm{ref}}), \label{eq:Href-Lambdaref}
\end{equation}
determine piecewise constant profiles
\begin{equation*}
  \Lambda^{\mathrm{ref}}(\xi) =  \Lambda_j^{\mathrm{ref}}, \quad  H^{\mathrm{ref}}(\xi) =  H_j^{\mathrm{ref}}, \quad  \text{for }\ \sigma_{j-1}<\xi<\sigma_j \quad (j=1,\ldots,M).
\end{equation*}
The corresponding profiles of the fan are
\begin{equation}
  \Lambda^{\mathrm{fan}}(\xi)  =
  \begin{cases}
    \Lambda_-,      &\xi<\nu_0,\\[1mm]
    \Lambda_i,      &\nu_{i-1}<\xi<\nu_i,\\[1mm]
    \Lambda_+,      &\xi>\nu_N,
  \end{cases}
 \quad  H^{\mathrm{fan}}(\xi) =
  \begin{cases}
    H_-,      &\xi<\nu_0,\\[1mm]
    H_i,      &\nu_{i-1}<\xi<\nu_i,\\[1mm]
    H_+,      &\xi>\nu_N,
  \end{cases}\label{eq:Hfan-Lambdafan}
\end{equation}
as defined in \eqref{eq:fanprof1}, \eqref{eq:fanprof2}, with $i=1,\ldots,N$.

We assume throughout that the reference and fan subsolutions correspond to the same Riemann data, so that
$  H^{\mathrm{ref}}(\xi)= H^{\mathrm{fan}}(\xi)$
and
$\Lambda^{\mathrm{ref}}(\xi) =  \Lambda^{\mathrm{fan}}(\xi)$
for all sufficiently large $|\xi|$.\\

For the energy we compute using \eqref{eq:renormalized-energy-difference}:
\begin{theorem}
\label{thm:finite-dimensional-energy}
Let $ (\rho,v)$
be any convex integration realization generated by the fan subsolution,
and let $(\rho^{\mathrm{ref}},v^{\mathrm{ref}})$ be a piecewise constant self-similar reference solution with the same
Riemann data. Then
\begin{equation}
   \Energy(\rho,v,t) - \Energy(\rho^{\mathrm{ref}},v^{\mathrm{ref}},t) = \ell t\Ecoef,  \quad \text{where}\ \ \Ecoef :=  \int_{\R}  \left[        H^{\mathrm{fan}}(\xi)  - H^{\mathrm{ref}}(\xi) \right]d\xi.   \label{eq:energy-difference-general}
\end{equation}
independently of $\ell>0$ and $t>0$, the fan has strictly smaller energy rate than the reference solution if and only if $\Ecoef<0$, and strictly larger energy rate if and only if $\Ecoef>0$.
\end{theorem}
\begin{proof}
By definition of the energy difference,
\begin{align*}
    \Energy(\rho,v,t) - \Energy(\rho^{\mathrm{ref}},v^{\mathrm{ref}},t)
  =  \int_{\T_\ell} \int_{\R}\left[   H^{\mathrm{fan}}\left(      \frac{x_2}{t}\right)    -  H^{\mathrm{ref}}    \left( \frac{x_2}{t}    \right) \right] dx_2\,dx_1.
\end{align*}
Integration in $x_1$ contributes a factor $\ell$. For fixed
$t>0$, the change of variables $ x_2=t\xi$ leads to
$$\Energy(\rho,v,t) - \Energy(\rho^{\mathrm{ref}},v^{\mathrm{ref}},t)
 =  \ell t  \int_{\R}  \left[   H^{\mathrm{fan}}(\xi)-   H^{\mathrm{ref}}(\xi) \right]d\xi.$$
On allowing $t$ to vary, the required rates are obtained by time differentiation.
\end{proof}
An analogous computation for the action shows, using \eqref{eq:renormalized-action-difference}:
\begin{theorem}
\label{thm:finite-dimensional-action}
Let $ (\rho,v)$
be any convex integration realization generated by the fan subsolution,
and let $(\rho^{\mathrm{ref}},v^{\mathrm{ref}})$ be a piecewise constant self-similar reference solution with the same
Riemann data. Then
\begin{equation}
  \A(\rho^{\mathrm{ref}},v^{\mathrm{ref}},t)  -  \A(\rho,v,t) = \frac{\ell t^2}{2}\Dcoef,  \quad \text{where}\ \ \Dcoef :=  \int_{\R}  \left[    \Lambda^{\mathrm{ref}}(\xi) -    \Lambda^{\mathrm{fan}}(\xi)  \right]d\xi.   \label{eq:action-difference-general}
\end{equation}
In particular, the fan has strictly smaller action than the reference solution if and only if $\Dcoef>0$, and a strictly larger action if and only if $\Dcoef<0$, independently of $\ell>0$ and $t>0$.
\end{theorem}
It remains to compute $\Ecoef$ and $\Dcoef$ from the data of the fan and the reference solution.
\begin{proposition} 
\label{prop:discrete-overlap-formulas}
Assume that the fan and reference profiles agree on $P_-$ and on $P_+$. Then 
\begin{equation*}
    \Ecoef =  \sum_{i=1}^{N}  \sum_{j=1}^{M} \omega_{ij}  \left(    H_i-H_j^{\mathrm{ref}} \right), \qquad \Dcoef  = \sum_{i=1}^{N}  \sum_{j=1}^{M} \omega_{ij}  \left(  \Lambda_j^{\mathrm{ref}}-\Lambda_i
  \right),
 \end{equation*}
with $\omega_{ij}  := \left[\min\{\nu_i,\sigma_j\}-   \max\{\nu_{i-1},\sigma_{j-1}\}  \right]_+$, where $[z]_+:=\max\{z,0\}$.
\end{proposition}

\begin{proof} First note  that $\omega_{ij}$ is the length of the interval $(\nu_{i-1},\nu_i) \cap  (\sigma_{j-1},\sigma_j)$.  
On every nonempty such interval the fan and reference densities are constant. The contribution of this interval to $\Ecoef$ is therefore $\omega_{ij} \left(    H_i-H_j^{\mathrm{ref}} \right)$. Similarly, the contribution of the same interval to $\Dcoef$
is $\omega_{ij}\left(    \Lambda_j^{\mathrm{ref}}-\Lambda_i \right)$.
Summation over all overlaps proves both identities.
\end{proof}

We now apply these results to the classical contact discontinuity. The classical contact solution is defined in terms of 
\begin{equation*}
  \rho_-=\rho_+=:\rho_0,
  \qquad
  v_{-2}=v_{+2}=:s,
\end{equation*}
and $ v_{-1}\neq v_{+1}$ as 
\begin{equation}
  (\rho^{\mathrm c},v^{\mathrm c})(x,t)  =
  \begin{cases}
    (\rho_0,v_-),      &x_2<st,\\[1mm]
    (\rho_0,v_+),      &x_2>st.
  \end{cases}
  \label{eq:classical-contact-solution}
\end{equation}
Define the corresponding mechanical energy and Lagrangian densities by
\begin{align*}
  H_-^{\mathrm c} &:=   \frac12\rho_0|v_-|^2+  \rho_0\eint(\rho_0),
  \qquad   H_+^{\mathrm c}:=  \frac12\rho_0|v_+|^2 +  \rho_0\eint(\rho_0),\\
  \Lambda_-^{\mathrm c}  &:=  \frac12\rho_0|v_-|^2 -  \rho_0\eint(\rho_0), \qquad  \Lambda_+^{\mathrm c}
  :=  \frac12\rho_0|v_+|^2-  \rho_0\eint(\rho_0).
\end{align*}
\begin{proposition}
\label{prop:fan-contact-coefficients}
Let $ (\rho,v)$
be any convex integration realization generated by the fan subsolution, and let $(\rho^{\mathrm{ref}},v^{\mathrm{ref}}) = (\rho^{\mathrm c},v^{\mathrm c})$ be the classical contact solution. Suppose that $\nu_{k-1}<s<\nu_k$ for some $k\in\{1,\ldots,N\}$. Then 
\begin{align*}
  \Ecoef_{\mathrm c}  &=  \sum_{i=1}^{k-1}  (\nu_i-\nu_{i-1}) \left(    H_i-H_-^{\mathrm c}  \right)+ (s-\nu_{k-1}) \left(   H_k-H_-^{\mathrm c} \right)\\ & \qquad  +  (\nu_k-s)
  \left( H_k-H_+^{\mathrm c}  \right)+  \sum_{i=k+1}^{N}  (\nu_i-\nu_{i-1}) \left(   H_i-H_+^{\mathrm c}  \right),\\
  \Dcoef_{\mathrm c}  
&=  \sum_{i=1}^{k-1} (\nu_i-\nu_{i-1})  \left(    \Lambda_-^{\mathrm c}-\Lambda_i \right)+
  (s-\nu_{k-1})  \left(   \Lambda_-^{\mathrm c}-\Lambda_k  \right)\\ 
& \qquad + (\nu_k-s)  \left(   \Lambda_+^{\mathrm c}-\Lambda_k\right)
  +\sum_{i=k+1}^{N}  (\nu_i-\nu_{i-1})\left(    \Lambda_+^{\mathrm c}-\Lambda_i  \right).
\end{align*}
\end{proposition}

\begin{proof}
This follows directly from Proposition \ref{prop:discrete-overlap-formulas}. All fan intervals left of $s$ are compared with the left reference state, all
intervals right of $s$ with the right reference state, and the $k$-th fan interval is divided at $s$.
\end{proof}

\section{Comparison and stability of admissibility criteria} \label{sec:comparisonstability}

In this section we discuss the contributions from the kinetic and the
internal energy to  $\Ecoef$ and $\Dcoef$. The results are used to clarify the relationship between the action and entropy rate admissibility criteria. We then discuss their stability under parameter perturbations.

As in Section \ref{sec:finite-dimensional-formulas}, we let $ (\rho,v)$ be any convex integration realization generated by a given fan subsolution and $(\rho^{\mathrm{ref}},v^{\mathrm{ref}})$ a piecewise constant self-similar reference solution with the same Riemann data. Setting $ C_j^{\mathrm{ref}}:= |v_j^{\mathrm{ref}}|^2$ for the reference solution and the kinetic level $C_i$ for the fan solution, we define kinetic and internal energy mismatches as
\begin{equation}
  \Kcoef  :=  \sum_{i=1}^{N} \sum_{j=1}^{M}  \omega_{ij}  \left(    \rho_j^{\mathrm{ref}}C_j^{\mathrm{ref}}   -    \rho_iC_i  \right), \quad   \Icoef
  :=  \sum_{i=1}^{N}  \sum_{j=1}^{M}  \omega_{ij}  \left[    \rho_i\eint(\rho_i)   -    \rho_j^{\mathrm{ref}}
    \eint(\rho_j^{\mathrm{ref}})  \right].
  \label{eq:mismatches}
\end{equation}
A positive value of $\Kcoef$ means that, after weighting by the fan geometry, the reference solution has larger kinetic energy density than the fan.\\

Straightforward algebra, using the identities \eqref{eq:Href-Lambdaref}, \eqref{eq:Hfan-Lambdafan} (for each $(i,j)$) 
\begin{align*}
  \Lambda_j^{\mathrm{ref}}-\Lambda_i
  & =  \frac12  \left(    \rho_j^{\mathrm{ref}}C_j^{\mathrm{ref}}    -    \rho_iC_i \right)+  \left[    \rho_i\eint(\rho_i)  -    \rho_j^{\mathrm{ref}}    \eint(\rho_j^{\mathrm{ref}})  \right],\\
  H_i-H_j^{\mathrm{ref}}
  &=  -\frac12 \left(    \rho_j^{\mathrm{ref}}C_j^{\mathrm{ref}}-   \rho_iC_i \right)+ \left[    \rho_i\eint(\rho_i)   -    \rho_j^{\mathrm{ref}}  \eint(\rho_j^{\mathrm{ref}}) \right],
\end{align*}
and Proposition \ref{prop:discrete-overlap-formulas} shows an exact relation between the action and entropy rate criteria, in terms of kinetic and internal energy mismatches:
\begin{theorem} \label{thm:kinetic-internal-decomposition}
The action and entropy rate coefficients satisfy
\begin{equation}
 \Ecoef=-\frac12\Kcoef+\Icoef,\qquad   \Dcoef  =  \frac12\Kcoef+\Icoef.
  \label{eq:main-KI-decomposition}
\end{equation}
\end{theorem}

The following Corollary confirms the interpretation in Figure \ref{fig:KI-phase-diagram} from the Introduction. 
\begin{corollary}
\label{cor:agreement-condition}
$(i)$ The action and entropy rate criteria both strictly prefer the fan to the reference solution if and
only if  $\Kcoef>0$ and  $|\Icoef|<\frac{\Kcoef}{2}$.\\
$(ii)$ Both criteria strictly prefer the reference solution if and only if $\Kcoef<0$ and  $|\Icoef|<\frac{|\Kcoef|}{2}$.\\
$(iii)$ The action rate criterion strictly prefers the fan while the entropy rate criterion strictly prefers  the reference solution if and only if $\Icoef>\frac{|\Kcoef|}{2}$.\\
 $(iv)$ The action rate criterion strictly prefers the reference solution while the entropy rate criterion strictly prefers  the fan if and only if $\Icoef<-\frac{|\Kcoef|}{2}$.
\end{corollary}
\begin{proof}
To show $(i)$, we use Theorem \ref{thm:kinetic-internal-decomposition} and recall that the entropy rate criterion strictly prefers the fan if and only if $  \Ecoef=-\frac12\Kcoef+\Icoef<0$ (Theorem \ref{thm:finite-dimensional-energy}), while
the action rate criterion strictly prefers the fan if and only if $\Dcoef=\frac12\Kcoef+\Icoef>0$ (Theorem \ref{thm:finite-dimensional-action}). These inequalities for $\Ecoef, \Dcoef$ readily translate into the stated conditions on $\Kcoef, \Icoef$.\\
Parts $(ii)$, $(iii)$ and $(iv)$ are obtained analogously.
\end{proof}

\noindent \textbf{Contact discontinuity.} We now show that admissible fan competitors to a planar contact discontinuity are always preferred by both action and entropy rate criteria, as in Part (i) of Corollary \ref{cor:agreement-condition}. In particular, this applies to the Krupa--Sz\'ekelyhidi and Horimoto solutions studied below.

We study a classical planar contact discontinuity with
\begin{equation}
  \rho_-=\rho_+=\rho_0, \qquad v_{-2}=v_{+2}=s.
  \label{eq:universal-contact-data}
\end{equation}
Let $ (\bar\rho,\bar v,\bar u)$ be an admissible piecewise constant fan subsolution have the same Riemann data and exterior states. The fan regions $P_-,P_1,\ldots,P_N,P_+$ are separated by $ \nu_0<\nu_1<\cdots<\nu_N$. For any piecewise constant fan quantity having values $  f_-,f_1,\ldots,f_N,f_+$, we define its jump at the interface $x_2=\nu_rt$ by
\begin{equation*}
[f]_0:=f_1-f_-,  \qquad  [f]_r:=f_{r+1}-f_r  \quad (1\leq r\leq N-1),  \qquad [f]_N:=f_+-f_N. 
\end{equation*}
Further, define the fluxes
\begin{equation*}  \mathcal J_\pm:=\rho_\pm(v_\pm)_2, \, \  \mathcal J_i:=\rho_i(v_i)_2, \ \  \mathcal F_\pm :=  \bigl(H_\pm+p(\rho_\pm)\bigr)(v_\pm)_2,  \ \   \mathcal F_i  :=  \bigl(H_i+p(\rho_i)\bigr)(v_i)_2,
\end{equation*}
for  $i=1,\ldots,N$. The mass Rankine-Hugoniot relations and the entropy inequalities take the form
\begin{equation}
  [\mathcal J]_r-\nu_r[\rho]_r=0,  \qquad r=0,\ldots,N,
  \label{eq:contact-fan-mass-RH}
\end{equation}
and
\begin{equation}
  [\mathcal F]_r-\nu_r[H]_r\leq0, \qquad r=0,\ldots,N.  \label{eq:contact-fan-entropy-RH}
\end{equation}

\begin{theorem}
\label{thm:contact-agreement}
Assume that $p\in C^1(0,\infty)$ and $p'(\rho)>0$. Let an admissible piecewise constant fan subsolution have the same Riemann data as the classical planar contact discontinuity \eqref{eq:universal-contact-data}. Then the mismatch and rate coefficients defined in Sections \ref{sec:finite-dimensional-formulas} and \ref{sec:comparisonstability} satisfy
\begin{equation*}
  \Icoef\geq0,  \qquad  \Ecoef\leq0.
\end{equation*}
Consequently,
\begin{equation*}
  \Kcoef=2(\Icoef-\Ecoef)\geq2\Icoef\geq0,  \qquad  \Dcoef=2\Icoef-\Ecoef\geq0.
\end{equation*}
Moreover,  $\Icoef>0$ whenever the fan density is not identically equal to  $\rho_0$.  If at least one interface entropy inequality  \eqref{eq:contact-fan-entropy-RH} is strict, then
  \begin{equation*}
    \Ecoef<0,    \qquad  \Kcoef>2\Icoef\geq0,   \qquad   \Dcoef>0.   
  \end{equation*}
\end{theorem}
Therefore, every strictly dissipative admissible fan competitor to a planar contact lies strictly in the right hand sector of Figure \ref{fig:KI-phase-diagram}, and both the entropy rate and action rate criteria strictly prefer the fan to the classical contact discontinuity.
\begin{proof}
Summing the mass Rankine--Hugoniot relations \eqref{eq:contact-fan-mass-RH} over the $N+1$ fan interfaces, we obtain
\begin{equation*}  \sum_{r=0}^{N}\nu_r[\rho]_r =  \sum_{r=0}^{N}[\mathcal J]_r =  \mathcal J_+-\mathcal J_-. 
\end{equation*}
For the contact data,  $\mathcal J_-=\mathcal J_+=\rho_0s$, and hence
\begin{equation}
  \sum_{r=0}^{N}\nu_r[\rho]_r=0.  \label{eq:contact-fan-summed-mass}
\end{equation}

Let $\rho^{\mathrm{fan}}(\xi)$ denote the fan density profile in the similarity variable $\xi=\frac{x_2}{t}$. It equals $\rho_0$ outside the interval $(\nu_0,\nu_N)$. Since its distributional derivative is
\[  \frac{d}{d\xi}\rho^{\mathrm{fan}} =  \sum_{r=0}^{N}[\rho]_r\delta_{\nu_r}, \]
integration by parts gives
\begin{align*}
\int_{\mathbb R}   \bigl(    \rho^{\mathrm{fan}}(\xi)-\rho_0  \bigr)\,d\xi  &=  -\sum_{r=0}^{N}\nu_r[\rho]_r=0.
\end{align*}
Choose $R>0$ sufficiently large that $  (\nu_0,\nu_N)\subset(-R,R)$. Then
\begin{equation} \label{eq:rho-fan}
  \frac1{2R}  \int_{-R}^{R}  \rho^{\mathrm{fan}}(\xi)\,d\xi  =  \rho_0.
\end{equation}
We recall the pressure potential $P(\rho):=\rho e(\rho)$. Using $p(\rho)=\rho^2e'(\rho)$, one obtains $P''(\rho)=\frac{p'(\rho)}{\rho}$. Hence, $P$ is strictly convex, since $p'(\rho)>0$. Together with \eqref{eq:rho-fan} and Jensen's inequality we obtain
\[  \frac1{2R}  \int_{-R}^{R}  P\bigl(\rho^{\mathrm{fan}}(\xi)\bigr)\,d\xi  \geq  P(\rho_0).\]
Since the reference contact has the constant density $\rho_0$,
$\Icoef  =  \int_{\mathbb R}  \left[    P\bigl(\rho^{\mathrm{fan}}(\xi)\bigr) -    P(\rho_0) \right]d\xi  \geq0. $
Strict Jensen convexity now gives $\Icoef>0$ whenever $\rho^{\mathrm{fan}}\not\equiv\rho_0$.

We next compute the relative energy-rate coefficient. Let $H^{\mathrm{fan}}$ and $H^{\mathrm c}$ denote the fan and contact
energy profiles. Since their difference has compact support, integration by parts leads to
\begin{equation}  \Ecoef= s(H_+-H_-)  -  \sum_{r=0}^{N}\nu_r[H]_r.  \label{eq:contact-E-jump-formula}
\end{equation}

The exterior density, pressure, and normal velocity agree on both sides of the contact. Hence
$\mathcal F_+-\mathcal F_-=  s(H_+-H_-)$. Using this identity in \eqref{eq:contact-E-jump-formula}, we obtain
\begin{align*}
  \Ecoef   =  \mathcal F_+-\mathcal F_- -  \sum_{r=0}^{N}\nu_r[H]_r=  \sum_{r=0}^{N}  \left(    [\mathcal F]_r-\nu_r[H]_r \right)  \leq0
\end{align*}
by the interface entropy inequalities \eqref{eq:contact-fan-entropy-RH}. If at least one of these inequalities is strict, then $\Ecoef<0$.

Finally, Theorem \ref{thm:kinetic-internal-decomposition} gives
$\Kcoef=2(\Icoef-\Ecoef)$,
$\Dcoef=2\Icoef-\Ecoef$.
The conclusions now follow from $\Icoef\geq0$ and $\Ecoef\leq0$.
\end{proof}

\begin{corollary}
\label{cor:no-contact-disagreement}
An admissible piecewise constant fan competitor to a classical planar contact cannot lie in either disagreement sector of the $(\Kcoef,\Icoef)$-plane. Its coordinates belong to the  cone
\begin{equation*}
  \Icoef\geq0,  \qquad   \Kcoef\geq2\Icoef.  
\end{equation*}
If at least one interface entropy inequality is strict, then  $\Kcoef>2\Icoef\geq0$.
\end{corollary}

%The proof uses only the mass Rankine--Hugoniot relations, strict convexity of the pressure potential, and the interface entropy inequalities. 
The proof also applies to relaxed fans  when the normal entropy flux is an independent variable, as in Horimoto's formulation, provided that $\mathcal F_i$ is replaced by the corresponding prescribed normal relaxed entropy flux.\\

\noindent \textbf{The Chiodaroli--Kreml two-shock family.} The disagreement parts of Corollary \ref{cor:agreement-condition} are not empty. Part $(iv)$ is realized by the two-shock solutions of Chiodaroli and Kreml \cite{ChiodaroliKreml}, through results proved separately in \cite{ChiodaroliKreml} and \cite{GKKS1}. Example \ref{exampleKS} below specifically describes the curve CK of these solutions in Figure \ref{fig:KI-phase-diagram}, which crosses parts $(ii)$ and $(iv)$.

Consider Riemann data with $p(\rho)=\rho^\gamma$, no tangential velocity component, and a self-similar solution consisting of two
admissible shocks with intermediate state $(\rho_m,v_m)$ and speeds $\nu_-(\rho_m)<\nu_+(\rho_m)$. The fan subsolutions of
\cite{ChiodaroliKreml} have a single turbulent region with density $\rho_1$, kinetic level
\begin{equation*}
C(\rho_1)=\beta^2(\rho_1)+\epsilon_1(\rho_1)+\epsilon_2,  \qquad  \beta^2(\rho_m)=v_m^2,\quad \epsilon_1(\rho_m)=0,\quad \epsilon_2>0,
\end{equation*}
and interface speeds $\nu_-(\rho_1)<\nu_+(\rho_1)$ given by the same functions as the classical shock speeds, evaluated at $\rho_1$ instead
of $\rho_m$. We refer to \cite[Section~4]{GKKS1} for this summary and \cite[(4.46)--(4.54), (4.82), (4.83)]{ChiodaroliKreml} for the
explicit formulas and the admissible parameter ranges. The reference partition and the fan partition thus differ only in the position of the two interfaces, and the coefficients $\Dcoef$, $\Ecoef$, $\Kcoef$, $\Icoef$ of the comparison are given by the overlap formulas of
Proposition~\ref{prop:discrete-overlap-formulas}. When the fan wedge is contained in the reference wedge, $\nu_-(\rho_m)\le\nu_-(\rho_1)<\nu_+(\rho_1)\le\nu_+(\rho_m)$, they read
\begin{align}
  \Kcoef(\rho_1)   =&\bigl(\nu_-(\rho_1)-\nu_-(\rho_m)\bigr) \bigl(\rho_m v_m^2-\rho_- v_{b-}^2\bigr)
     +\bigl(\nu_+(\rho_1)-\nu_-(\rho_1)\bigr)     \bigl(\rho_m v_m^2-\rho_1 C(\rho_1)\bigr)
  \nonumber\\
  &+\bigl(\nu_+(\rho_m)-\nu_+(\rho_1)\bigr) \bigl(\rho_m v_m^2-\rho_+ v_{b+}^2\bigr),
  \label{eq:CK-K}\\[2pt]   \Icoef(\rho_1)
  ={}&\bigl(\nu_-(\rho_1)-\nu_-(\rho_m)\bigr)      \bigl(\rho_-\eint(\rho_-)-\rho_m\eint(\rho_m)\bigr)
     +\bigl(\nu_+(\rho_1)-\nu_-(\rho_1)\bigr)      \bigl(\rho_1\eint(\rho_1)-\rho_m\eint(\rho_m)\bigr)
  \nonumber\\
  &+\bigl(\nu_+(\rho_m)-\nu_+(\rho_1)\bigr)   \bigl(\rho_+\eint(\rho_+)-\rho_m\eint(\rho_m)\bigr),
  \label{eq:CK-I}
\end{align}
with the general ordering handled by the overlap weights of Proposition~\ref{prop:discrete-overlap-formulas}. At the endpoint $\rho_1=\rho_m$ the two partitions coincide and the comparison is exactly computable.

\begin{proposition}
\label{prop:CK-endpoint}
At $\rho_1=\rho_m$, both criteria prefer the classical two-shock solution:
\begin{align*} 
 \Kcoef(\rho_m)&  =\Bigl(-\bigl(\nu_+(\rho_m)-\nu_-(\rho_m)\bigr)\rho_m\epsilon_2<0,\qquad \Icoef(\rho_m) = 0,\\
\Dcoef(\rho_m)   &=-\tfrac12\bigl(\nu_+(\rho_m)-\nu_-(\rho_m)\bigr)\rho_m\epsilon_2<0, \qquad
\Ecoef(\rho_m)   =\tfrac12\bigl(\nu_+(\rho_m)-\nu_-(\rho_m)\bigr)\rho_m\epsilon_2>0.
\end{align*}
\end{proposition}

\begin{proof}
Only the middle overlap term survives in \eqref{eq:CK-K}, \eqref{eq:CK-I}. Since $C(\rho_m)=v_m^2+\epsilon_2$, the kinetic mismatch is $\rho_m v_m^2-\rho_m C(\rho_m)=-\rho_m\epsilon_2$, while the internal energy mismatch vanishes because the densities agree. The values of $\Dcoef$, $\Ecoef$ follow from the decomposition $\Dcoef=\tfrac12\Kcoef+\Icoef$, $\Ecoef=-\tfrac12\Kcoef+\Icoef$, and recover \cite[(4.15)]{GKKS1}.
\end{proof}

From \cite[Theorem~2]{ChiodaroliKreml} and \cite[Theorem~6]{GKKS1} we know that there exist solutions in this family with $\Ecoef<0$ and   $\Dcoef<0$, i.e., where the entropy rate criterion prefers the convex integration solutions while the least action principle prefers the classical two-shock solution. They are realized as $\rho_1$ moves away from $\rho_m$.

\begin{example}\label{exampleKS}
For fixed data and $\epsilon_2$, the map $\rho_1\mapsto(\Kcoef(\rho_1),\Icoef(\rho_1))$ traces a curve in Figure \ref{fig:KI-phase-diagram} which starts, by Proposition~\ref{prop:CK-endpoint}, on the negative $\Kcoef$-axis. We consider symmetric two-shock Riemann data $v_\pm=(0,\mp V)$ with $\gamma=2$, $\rho_-=\rho_+=1$ and $V^2=891/10$, for which the intermediate state is $(\rho_m,v_m)=(10,0)$, and take $\epsilon_2=1/20$, with $\rho_1$ in $[6.4,\rho_m)$, $\rho_1$ decreasing away from the origin. By \cite[Theorem~2]{ChiodaroliKreml} the curve crosses the boundary $\Icoef=\Kcoef/2$, on which $\Ecoef=0$ and the Chiodaroli--Kreml phenomenon sets in, while by \cite[Theorem~6]{GKKS1}, for $\gamma=2$, it never reaches the other boundary $\Icoef=-\Kcoef/2$, on which $\Dcoef=0$. In contrast to the
Krupa-Sz\'ekelyhidi fan below, where the kinetic mismatch dominates and the two criteria agree, along the Chiodaroli--Kreml family the internal energy mismatch takes over and the criteria separate.  
\end{example}

\noindent \textbf{Markfelder–Pellhammer shock-contact data.}
A further point in the agreement region is provided by the self-similar one-wedge construction of Markfelder and Pellhammer \cite[Sections~2.2--2.5]{MP2}. For their data $p(\rho)=\rho^2$, the one-dimensional reference solution consists of a shock, a contact discontinuity, and a second shock, whereas the convex-integration solution contains one turbulent fan region.

\begin{proposition} \label{prop:MP-point}
For the Markfelder-Pellhammer solution, we obtain
\begin{align}
 \Kcoef_{\mathrm{MP}} &=  \frac{    2174152029300\sqrt{14}-1363173585011\sqrt{35}}{ 535964450  },\quad  
  \Icoef_{\mathrm{MP}}  &= \frac{ 60\sqrt{35}-150\sqrt{14}}7.  \label{eq:MP-I}
\end{align}
Consequently,
\begin{align}
  \Dcoef_{\mathrm{MP}}  &=  \frac{    2151182124300\sqrt{14}  -    1353985623011\sqrt{35}}{    1071928900 }  >0,  \label{eq:MP-D}  \\ 
 \Ecoef_{\mathrm{MP}}  &=\frac{    1372361547011\sqrt{35} -    2197121934300\sqrt{14} }{    1071928900 }  <0.  \label{eq:MP-E}
\end{align}
In particular, $\Kcoef_{\mathrm{MP}}>0$, $ -\frac{\Kcoef_{\mathrm{MP}}}{2} <  \Icoef_{\mathrm{MP}} <0$.
\end{proposition}
In Figure \ref{fig:KI-phase-diagram} this point lies in the lower part of the right hand sector: both the action rate and entropy rate criteria prefer the convex-integration solutions to the one-dimensional reference solution, but it lies outside the shaded region of contact solutions in Theorem \ref{thm:contact-agreement}.
\begin{proof}
The fan and reference speeds satisfy
$  \mu_0<\sigma_-<0<\sigma_+<\mu_1$,
with overlap widths
\[  \sigma_--\mu_0 =  \frac{25\sqrt{14}-15\sqrt{35}}7,  \qquad  \sigma_+-\sigma_- =  \frac{5\sqrt{14}}7,\]
and
\[  \mu_1-\sigma_+  =  \frac{20\sqrt{35}-30\sqrt{14}}7.\]
For $p(\rho)=\rho^2$, the pressure potential is $P(\rho)=\rho^2$. The reference kinetic energy densities $\rho|v|^2$ are $ \frac{540}{7}$, $252$, $\frac{729}{7}$, whereas the turbulent kinetic energy density is $5C_1$. The corresponding pressure-potential densities are $1$, $49$, $4$ and $25$, respectively. Substitution into the overlap formulas gives \eqref{eq:MP-I}. Equations \eqref{eq:MP-D}, \eqref{eq:MP-E} then follow from 
$  \Dcoef_{\mathrm{MP}} =  \frac12\Kcoef_{\mathrm{MP}}+\Icoef_{\mathrm{MP}}$,
$\Ecoef_{\mathrm{MP}}  =  -\frac12 \Kcoef_{\mathrm{MP}}+\Icoef_{\mathrm{MP}}$.
\end{proof}

\noindent \textbf{Parameter dependence.} The application to the classical contact discontinuity  involves numerically computed solutions \cite{KrupaSzekelyhidi}. Therefore we require quantitative stability results for the admissibility criteria under perturbations of the parameters. We let $U\subset \R^d$ be a neighborhood of the vector $  y\in \R^d$ of the parameters describing the fan and reference states, and assume that the ordering of all wave speeds and the
overlap pattern do not change in $U$.

\begin{proposition}
\label{prop:smooth-parameter-dependence}
Suppose that
$  y\mapsto(\nu_i(y),\rho_i(y),C_i(y),\eint_y(\rho_i(y)))$
and the analogous map for the reference solution are $C^1$ in $U$. Then
$ y\mapsto(\Dcoef(y),\Ecoef(y),\Kcoef(y),\Icoef(y))$ is $C^1$ in $U$. In particular, strict preference by the action or entropy rate criteria is an open condition in $y$.
\end{proposition}
\begin{proof}
When the wave ordering is fixed, every $\omega_{ij}$ is an affine combination of wave speeds. Hence, each coefficient is a finite sum
of products of continuously differentiable functions. The final assertion follows by continuity.
\end{proof}

\section{Application to the Krupa-Sz\'ekelyhidi fan} \label{sec:KS-application}

We now specialize the general theory to the  solutions constructed by Krupa and Sz\'ekelyhidi \cite{KrupaSzekelyhidi}. Their construction assumes Riemann data that for the one-dimensional self-similar reference solution is a planar contact discontinuity. For a specially constructed smooth pressure law, they produce a strict admissible fan subsolution with three turbulent regions. Convex integration then yields infinitely many genuinely two-dimensional
bounded entropic weak solutions.

The relevant Krupa and Sz\'{e}kelyhidi (KS) fan consists of five regions:
\begin{align*}
  &P_- =  \{x_2<\nu_-t\}, \qquad P_+ =\{x_2>\nu_+t\},  \\
  &P_1=  \{\nu_-t<x_2<\nu_1t\}, \quad  P_2= \{\nu_1t<x_2<\nu_2t\}, \quad  P_3=  \{\nu_2t<x_2<\nu_+t\}.
\end{align*}
 separated by four speeds $\nu_-<\nu_1<\nu_2<\nu_+$. The exterior data satisfy $ \rho_-=\rho_+=:\rho_0$
and $  v_{-2}=v_{+2}=:s$. The wave speeds in the example of \cite{KrupaSzekelyhidi} satisfy $  \nu_-<\nu_1<s<\nu_2<\nu_+$. The contact line therefore crosses the central turbulent region.

As above, we denote by $ \rho_i$ the fan densities, by $C_i$ the kinetic energy levels and by  $e_i:=\eint(\rho_i)$ the specific internal energy values ($i=1,2,3$). Consequently, the  mechanical energy and Lagrangian densities are given by
\begin{equation}
   H_i^{\mathrm{KS}} :=  \rho_i  \left(  \frac{C_i}{2}+e_i \right),\quad \Lambda_i^{\mathrm{KS}}:=  \rho_i
  \left(    \frac{C_i}{2}-e_i \right)  \qquad  (i=1,2,3).
\end{equation}

Recall finally the classical contact discontinuity, as given by \eqref{eq:classical-contact-solution},
\begin{equation*}
  (\rho^{\mathrm c},v^{\mathrm c})(x,t) =
  \begin{cases}
    (\rho_0,v_-),      &x_2<st,\\[1mm]
    (\rho_0,v_+),      &x_2>st.
  \end{cases}
\end{equation*}
 The classical mechanical energy and Lagrangian densities are finally given as
\begin{align*}
  H_\pm^{\mathrm c}:=  \rho_0   \left[    \frac12    \left(      v_{\pm1}^2+s^2  \right) +    \eint(\rho_0)  \right],\quad   \Lambda_\pm^{\mathrm c}:=  \rho_0  \left[   \frac12 \left(    v_{\pm1}^2+s^2 \right) -    \eint(\rho_0) \right].
\end{align*}
We use the formulas from Proposition \ref{prop:fan-contact-coefficients}, specialized to  three turbulent regions with speeds
$\nu_-<\nu_1<\nu_2<\nu_+$ and $\nu_1<s<\nu_2$.
From Equations \eqref{eq:three-region-entropy-coefficient} and \eqref{eq:three-region-action-coefficient}, we obtain:
\begin{proposition}
\label{prop:KS-action-coefficient}
For every convex integration realization generated by the KS fan,
\begin{equation*}
   \Energy_{\mathrm{KS}}(t) -  \Energy_{\mathrm c}(t) =  \ell t\,\Ecoef_{\mathrm{KS}}, \quad \A_{\mathrm c}(t)-\A_{\mathrm{KS}}(t)  =  \frac{\ell t^2}{2} \Dcoef_{\mathrm{KS}},
\end{equation*}
where
\begin{align}
  \Ecoef_{\mathrm{KS}}
&  =  (\nu_1-\nu_-) \left(    H_1^{\mathrm{KS}}-H_-^{\mathrm c}  \right) +  (s-\nu_1)
  \left(    H_2^{\mathrm{KS}}-H_-^{\mathrm c}  \right) \nonumber \\ &\qquad +  (\nu_2-s)  \left(    H_2^{\mathrm{KS}}-H_+^{\mathrm c}   \right)+ (\nu_+-\nu_2)  \left(    H_3^{\mathrm{KS}}-H_+^{\mathrm c} \right),
  \label{eq:three-region-entropy-coefficient}\\
  \Dcoef_{\mathrm{KS}}
  &=  (\nu_1-\nu_-) \left(    \Lambda_-^{\mathrm c}-\Lambda_1^{\mathrm{KS}}
  \right)+  (s-\nu_1)  \left(    \Lambda_-^{\mathrm c}-\Lambda_2^{\mathrm{KS}}
  \right) \nonumber \\ &\qquad +(\nu_2-s) \left(    \Lambda_+^{\mathrm c}-\Lambda_2^{\mathrm{KS}}
  \right)+  (\nu_+-\nu_2)  \left(    \Lambda_+^{\mathrm c}-\Lambda_3^{\mathrm{KS}}
  \right).  \label{eq:three-region-action-coefficient}
\end{align}
All convex integration realizations generated by this fan have the same values of $\Ecoef_{\mathrm{KS}}, \Dcoef_{\mathrm{KS}}$.
\end{proposition}
\begin{proof}
The formulas are a special case of Proposition \ref{prop:fan-contact-coefficients}. Proposition \ref{prop:rigidity-fixed-fan} confirms the independence of the convex integration solution.
\end{proof}
Note that the classical contact discontinuity has no entropy production at the
contact interface and $\frac{d}{dt}\Energy_{\mathrm c}(t)=0$. We conclude $  \Ecoef_{\mathrm{KS}}<0$ if and only if the fan realizations strictly dissipate energy. This is confirmed by the exact evaluation below.

We similarly compute the kinetic and internal energy mismatches,  $\Kcoef_{\mathrm{KS}}, \Icoef_{\mathrm{KS}}$:
\begin{align}
  \Kcoef_{\mathrm{KS}}
  &=  (\nu_1-\nu_-)  \left(   \rho_0|v_-|^2-\rho_1C_1 \right)+  (s-\nu_1) \left(    \rho_0|v_-|^2-\rho_2C_2
  \right)\nonumber \\
 & \qquad +  (\nu_2-s) \left(    \rho_0|v_+|^2-\rho_2C_2  \right)+  (\nu_+-\nu_2)  \left(    \rho_0|v_+|^2-\rho_3C_3 \right),\\ 
 \Icoef_{\mathrm{KS}}
  &=  (\nu_1-\nu_-)  \left[    \rho_1e_1-\rho_0\eint(\rho_0) \right]+ (s-\nu_1)  \left[    \rho_2e_2-\rho_0\eint(\rho_0)  \right]\nonumber \\
  &\qquad +  (\nu_2-s)  \left[    \rho_2e_2-\rho_0\eint(\rho_0)  \right]+  (\nu_+-\nu_2)
  \left[    \rho_3e_3-\rho_0\eint(\rho_0)  \right] \label{eq:KS-kinetic-mismatch-explicit}\\
& =   (\nu_1-\nu_-)  \left[    \rho_1e_1-\rho_0\eint(\rho_0) \right]+  (\nu_2-\nu_1)  \left[    \rho_2e_2-\rho_0\eint(\rho_0)  \right]\nonumber\\& \qquad +  (\nu_+-\nu_2)  \left[    \rho_3e_3-\rho_0\eint(\rho_0) \right].\nonumber
\end{align}
Here, $ |v_\pm|^2=v_{\pm1}^2+s^2$, and the two middle terms in
\eqref{eq:KS-kinetic-mismatch-explicit} have been combined in the last step, since the classical density and internal energy agree
on both sides of the contact.

From Theorem \ref{thm:kinetic-internal-decomposition},
$  \Dcoef_{\mathrm{KS}}=  \frac12\Kcoef_{\mathrm{KS}}+  \Icoef_{\mathrm{KS}}$ and
$\Ecoef_{\mathrm{KS}} =  -\frac12\Kcoef_{\mathrm{KS}} +  \Icoef_{\mathrm{KS}}$.\\

By Theorem \ref{thm:contact-agreement}, the signs of the four coefficients are determined in advance for every admissible fan competitor to the classical contact discontinuity, and strictly so in the strictly dissipative case. The remaining task, undertaken in the next two sections, is to evaluate the four coefficients at the explicit parameters of \cite{KrupaSzekelyhidi}, thus locating the KS fan in Figure \ref{fig:KI-phase-diagram}, and to certify that these values and margins persist for the exact fan subsolution. 

\section{Stability and passage to the exact KS fan} \label{sec:KS-stability}

The construction of \cite{KrupaSzekelyhidi} proceeds in two stages: an explicit rational parameter vector specifies an approximate fan
subsolution, which is subsequently corrected into an exact admissible fan subsolution by a computer-assisted Newton argument. We do not
reproduce the existence argument. Instead, we record the rational data, evaluate the coefficients $\Dcoef_{\mathrm{KS}}$, $\Ecoef_{\mathrm{KS}}$, $\Kcoef_{\mathrm{KS}}$, $\Icoef_{\mathrm{KS}}$ there, and isolate the quantitative conditions needed to translate the resulting inequalities to the exact fan. This makes the admissibility calculation independent of the existence proof. The conditions are verified for the explicit fan of \cite{KrupaSzekelyhidi} in Subsection \ref{subsec:certification}. %Both action and entropy rate criteria therefore prefer the convex integration solutions associated with the exact fan to the classical contact discontinuity.
The conclusion that both criteria prefer the convex integration solutions associated with the exact fan to the classical contact
discontinuity follows already, unconditionally, from Theorem \ref{thm:contact-agreement} and is recorded in Corollary~\ref{cor:KS-unconditional} below. The certification quantifies it, by exact coefficient values and margins.

\subsection{Evaluation at the rational parameter vector} \label{sec:KS-rational-data}

In the first, computational stage of \cite{KrupaSzekelyhidi}, each turbulent region $P_i$ ($i=1,2,3$) carries a density $\rho_i>0$, a relaxed velocity $v_i=(\alpha_i,\beta_i)\in\R^2$, a trace-free symmetric stress
\begin{equation*}
  u_i = \begin{pmatrix} \gamma_i & \delta_i \\ \delta_i & -\gamma_i \end{pmatrix} \in S^{2\times 2}_0 ,
\end{equation*}
encoded by the pair $(\gamma_i,\delta_i)$, and a kinetic energy level $C_i>0$, while $e_i:=\eint(\rho_i)$ and $e_i':=\eint'(\rho_i)$ denote the values of the designed internal energy function and of its derivative at $\rho_i$. Together with the fan speeds $\nu_-<\nu_1<\nu_2<\nu_+$ and the exterior data $(\rho_0,v_\pm,s)$, these constants satisfy the subsolution and admissibility inequalities of Definition \ref{def:fan} strictly, whereas the equality constraints (the Rankine-Hugoniot conditions at the four interfaces) hold exactly at the first two interfaces and up to residuals below $10^{-11}$ at the two right-most interfaces.

Table \ref{tab:KS-rational-data} records the rational parameters entering the four coefficients. A hat denotes evaluation at these
rational values. All parameters produced by the numerical search are dyadic rationals $p/2^{k}$ (exact double-precision values). The two non-dyadic entries $\widehat\nu_-$, $\widehat\nu_1$ are the exact rational solutions of the Rankine-Hugoniot conditions at the first two interfaces, solved in \cite{KrupaSzekelyhidi} together with $\delta_1,\delta_2,e_1',e_2'$. The remaining parameters $\alpha_i,\beta_i,\gamma_i,\delta_i,e_i'$ ($i=1,2,3$) and $e_0'$ do not enter the four coefficients and are recorded in \cite{KrupaSzekelyhidi,KScode}.

\begin{table}[t]
  \centering\small
  \renewcommand{\arraystretch}{1.45}
  \begin{tabular}{@{}llll@{}}
    \hline
    $\widehat\rho_0=\tfrac{2708112612978501}{2^{48}}$ & $\approx 9.6211$ &
    $\widehat\rho_1=\tfrac{6811063536043807}{2^{49}}$ & $\approx 12.0989$ \\
    $\widehat\rho_2=\tfrac{2057060350258899}{2^{49}}$ & $\approx 3.6541$ &
    $\widehat\rho_3=\tfrac{3062207031116133}{2^{48}}$ & $\approx 10.8791$ \\[2pt]
    $\widehat e_0=-\tfrac{5041529442624971}{2^{41}}$ & $\approx -2292.62$ &
    $\widehat e_1=-\tfrac{5015532875605977}{2^{41}}$ & $\approx -2280.80$ \\
    $\widehat e_2=-\tfrac{5073206593829053}{2^{41}}$ & $\approx -2307.03$ &
    $\widehat e_3=-\tfrac{2515400677054201}{2^{40}}$ & $\approx -2287.74$ \\[2pt]
    $\widehat v_{-1}=-\tfrac{4098844157247653}{2^{46}}$ & $\approx -58.2481$ &
    $\widehat v_{+1}=\tfrac{3603433899522037}{2^{49}}$ & $\approx 6.4010$ \\
    $\widehat s=-\tfrac{996118042660627}{2^{46}}$ & $\approx -14.1557$ &
    $\widehat C_1=\tfrac{510415269881361}{2^{37}}$ & $\approx 3713.76$ \\
    $\widehat C_2=\tfrac{1515879700153707}{2^{41}}$ & $\approx 689.34$ &
    $\widehat C_3=\tfrac{1855257252703141}{2^{43}}$ & $\approx 210.92$ \\[2pt]
    $\widehat\nu_2=-\tfrac{4856156003780791}{2^{49}}$ & $\approx -8.6263$ &
    $\widehat\nu_+=\tfrac{7162856387903725}{2^{49}}$ & $\approx 12.7238$ \\[2pt]
    \multicolumn{3}{@{}l}{
      $\widehat\nu_-
       =-\tfrac{6486283176597739958874052307549}
               {196306040423407104692364247040}$}
    & $\approx -33.0417$ \\
    \multicolumn{3}{@{}l}{
      $\widehat\nu_1
       =-\tfrac{1153852086001065889673487658885}
               {60824224363690518566334889984}$}
    & $\approx -18.9703$ \\
    \hline
  \end{tabular}
  \caption{Rational parameters of the KS fan, as in the appendix of  \cite{KrupaSzekelyhidi}. Decimal approximations are included for orientation.}
  \label{tab:KS-rational-data}
\end{table}

In the second stage of \cite{KrupaSzekelyhidi}, the residual at the two right-most interfaces is removed by a Newton correction of the six
parameters
\begin{equation}
  x=(\alpha_3,\beta_3,\delta_3,\rho_3,\nu_+,\nu_2)\in\R^6,
  \label{eq:KS-parameter-vector}
\end{equation}
which are the only parameters entering the equality constraints at these interfaces. All remaining parameters retain their rational values. We denote the rational vector by
$\widehat x=(\widehat\alpha_3,\widehat\beta_3,\widehat\delta_3,
\widehat\rho_3,\widehat\nu_+,\widehat\nu_2)$, by $x^\ast$ 
the exact parameter vector, and by $\delta_{\mathrm{KS}}:=|x^\ast-\widehat x|$ the size of the correction. The results below require only an upper bound for $\delta_{\mathrm{KS}}$, not the unknown vector $x^\ast$.

Evaluation of the formulas of Proposition \ref{prop:KS-action-coefficient} at the data of Table \ref{tab:KS-rational-data} in exact rational arithmetic yields
\begin{align}
  \widehat{\Dcoef}_{\mathrm{KS}}
  &= 6609.677851759956\ldots,
  \qquad  \widehat{\Ecoef}_{\mathrm{KS}}  = -1407.036311937891\ldots,
  \label{eq:KS-rational-values}\\
  \widehat{\Kcoef}_{\mathrm{KS}}
  &= 8016.714163697847\ldots,  \qquad \widehat{\Icoef}_{\mathrm{KS}} = 2601.320769911032\ldots.
  \label{eq:KS-rational-KI-values}
\end{align}
In particular,
\begin{equation}  \label{eq:KS-rational-sign-bounds}
  \widehat{\Dcoef}_{\mathrm{KS}}>6609,  \qquad \widehat{\Ecoef}_{\mathrm{KS}}<-1407, 
\end{equation}
and
\begin{equation}  \label{eq:KS-rational-agreement}
  \widehat{\Kcoef}_{\mathrm{KS}}>0, \qquad  2\,\bigl|\widehat{\Icoef}_{\mathrm{KS}}\bigr|  <\widehat{\Kcoef}_{\mathrm{KS}},
\end{equation}
with agreement margin $\widehat{\Kcoef}_{\mathrm{KS}}/2 -|\widehat{\Icoef}_{\mathrm{KS}}| =1407.036311937891\ldots$, which coincides with $-\widehat{\Ecoef}_{\mathrm{KS}}$ because $\widehat{\Icoef}_{\mathrm{KS}}>0$. The strict inequalities \eqref{eq:KS-rational-sign-bounds}, \eqref{eq:KS-rational-agreement} have been verified symbolically, without floating-point evaluation, by the ancillary Maple worksheet described in Appendix \ref{sec:symbolic-verification}. The decimal expansions are only for illustration. Inequalities \eqref{eq:KS-rational-sign-bounds} imply:
\begin{proposition} \label{prop:KS-rational-selection} At the rational parameter vector of \cite{KrupaSzekelyhidi}, both the action and the entropy rate criterion prefer the convex integration fan to the classical contact discontinuity.
\end{proposition}
Accordingly, in the $(\Kcoef,\Icoef)$-plane of Figure \ref{fig:KI-phase-diagram} the rational vector lies inside the sector in which both criteria prefer the fan.

Proposition \ref{prop:KS-rational-selection} is not an instance of Theorem \ref{thm:contact-agreement}. At $\hat x$ the
Rankine-Hugoniot conditions at the interfaces $\nu_2$ and $\nu_+$ hold
only up to residuals below $10^{-11}$, so $\hat x$ is not a fan
subsolution in the sense of Definition~\ref{def:fan}, and the  identity \eqref{eq:contact-fan-summed-mass} from the proof of the general Theorem \ref{thm:contact-agreement} is not available. The inequalities  $\widehat{\Dcoef}_{\mathrm{KS}}>0$, $\widehat{\Ecoef}_{\mathrm{KS}}<0$ are therefore certified directly, independently of Theorem \ref{thm:contact-agreement}. The evaluation at $\hat x$ also provides the numerical location of the KS fan in Figure \ref{fig:KI-phase-diagram} and the size of the margins.

\subsection{Preservation of the fan geometry and rate coefficients}
\label{sec:KS-geometry}

By the construction of \cite{KrupaSzekelyhidi}, see also Theorem \ref{thm:deltacert}, the exact parameter vector $x^*$ defines a strict
admissible fan subsolution for contact Riemann data of the form \eqref{eq:universal-contact-data}. Theorem \ref{thm:contact-agreement} 
applies at $x^*$ and clarifies the qualitative comparison before any perturbation argument:

\begin{corollary}\label{cor:KS-unconditional}
For the exact KS fan,
\[
  D_{\mathrm{KS}}(x^*)>0,
  \qquad
  E_{\mathrm{KS}}(x^*)<0,
  \qquad
  0<I_{\mathrm{KS}}(x^*)<\tfrac12\,K_{\mathrm{KS}}(x^*).
\]
In particular, both the action and the entropy rate criteria strictly
prefer the convex integration solutions associated with the exact KS
fan to the classical contact discontinuity, the two criteria agree, and
the exact fan lies in the open right hand sector of
Figure \ref{fig:KI-phase-diagram}. 
\end{corollary}

\begin{proof}
The interface admissibility inequalities hold strictly at $x^*$: at
$\nu_-$ and $\nu_1$ they do not involve $x$ and hold strictly at $\hat x$, hence
unchanged at $x^*$, while at $\nu_2$ and $\nu_+$ they persist by
Theorem \ref{thm:deltacert}. Moreover $\rho_1$ retains its rational value,
so the fan density is not identically $\rho_0$. Parts (i) and (ii) of
Theorem \ref{thm:contact-agreement} give the displayed inequalities, and the
preference and agreement statements follow with
Corollary \ref{cor:agreement-condition}(i).
\end{proof}

The present subsection now transfers the certified values \eqref{eq:KS-rational-values}, \eqref{eq:KS-rational-KI-values} and their margins
from $\hat x$ to $x^*$.
The formulas of Proposition \ref{prop:KS-action-coefficient} require the ordering $\nu_1<s<\nu_2<\nu_+$, and only $\nu_2$ and $\nu_+$ vary
among the six corrected parameters $x$. The following result transfers the certified lower bound $D_{\mathrm{KS}}(\hat x)>6609$ to the exact fan. 
\begin{lemma}
\label{lem:KS-ordering-stability}
Assume
\begin{equation}   \label{eq:KS-ordering-condition}
  \delta_{\mathrm{KS}}   <\min\left\{    \widehat\nu_2-s,\;  \frac{\widehat\nu_+-\widehat\nu_2}{2}  \right\}.
\end{equation}
Then the ordering of the exact wave speeds is preserved,  $\nu_1<s<\nu_2^\ast<\nu_+^\ast$,
and the algebraic expressions for $\Dcoef_{\mathrm{KS}}$, $\Ecoef_{\mathrm{KS}}$ in Proposition \ref{prop:KS-action-coefficient} remain valid at $x^\ast$.
\end{lemma}

\begin{proof}
Since $|\nu_2^\ast-\widehat\nu_2|\leq\delta_{\mathrm{KS}}$, the first condition in \eqref{eq:KS-ordering-condition} gives $\nu_2^\ast>s$. Moreover,
\begin{equation*}
  \nu_+^\ast-\nu_2^\ast  \geq  \widehat\nu_+-\widehat\nu_2  -|\nu_+^\ast-\widehat\nu_+| -|\nu_2^\ast-\widehat\nu_2|  \geq  \widehat\nu_+-\widehat\nu_2-2\delta_{\mathrm{KS}},
\end{equation*}
which is positive by the second condition.
\end{proof}
For the data of Table \ref{tab:KS-rational-data}, the right hand side of \eqref{eq:KS-ordering-condition} equals $\widehat\nu_2-s=5.5294\ldots$.\\

Among the corrected variables, only $\rho_3$, $\nu_+$ and $\nu_2$ enter the action and entropy rate coefficients. Assume that the
thermodynamic interpolation used in the exact construction preserves $  \eint(\rho_3^\ast)=\widehat e_3$ and $C_3^\ast=\widehat C_3$.
Set $q_3:=\frac{\widehat C_3}{2}-\widehat e_3$,
so that the third fan Lagrangian density is $\Lambda_3^{\mathrm{KS}}(\rho_3)=q_3\rho_3$. Writing $\rho=\rho_3$, $\nu=\nu_+$, $\widetilde\nu=\nu_2$, the action coefficient becomes
\begin{align}
  \Dcoef_{\mathrm{KS}}(\rho,\nu,\widetilde\nu)
  =&(\widehat\nu_1-\widehat\nu_-)  \bigl(\Lambda_-^{\mathrm c}-\widehat\Lambda_1^{\mathrm{KS}}\bigr)  +(s-\widehat\nu_1)      \bigl(\Lambda_-^{\mathrm c}-\widehat\Lambda_2^{\mathrm{KS}}\bigr)   \notag\\
  &+(\widetilde\nu-s)      \bigl(\Lambda_+^{\mathrm c}-\widehat\Lambda_2^{\mathrm{KS}}\bigr) +(\nu-\widetilde\nu) \bigl(\Lambda_+^{\mathrm c}-q_3\rho\bigr), \label{eq:KS-variable-action-coefficient}
\end{align}
with nonzero derivatives
\begin{align}
  \partial_\rho\Dcoef_{\mathrm{KS}}  =-(\nu-\widetilde\nu)\,q_3,\ \  \partial_\nu\Dcoef_{\mathrm{KS}}  =\Lambda_+^{\mathrm c}-q_3\rho,\ \
  \partial_{\widetilde\nu}\Dcoef_{\mathrm{KS}}  =q_3\rho-\widehat\Lambda_2^{\mathrm{KS}}. \label{eq:KS-D-rho-derivative}
\end{align}

The next result transfers the certified lower bound $\Dcoef_{\mathrm{KS}}(\hat x)>6609$ to the exact fan.
\begin{proposition} \label{prop:KS-conditional-action}
Let $U$ be a convex neighbourhood containing the segment from
$\widehat x$ to $x^\ast$, and suppose $ \sup_{x\in U}|\nabla\Dcoef_{\mathrm{KS}}(x)|\leq L_D$. If
\begin{equation}\label{eq:KS-D-delta-condition}
  \delta_{\mathrm{KS}}<\frac{6609}{L_D},
\end{equation}
then  $\Dcoef_{\mathrm{KS}}(x^\ast)\geq \Dcoef_{\mathrm{KS}}(\hat x)-L_D\,\delta_{\mathrm{KS}}>0$, and the exact KS fan has smaller action than the classical contact discontinuity. 
\end{proposition}
\begin{proof}
From \eqref{eq:KS-rational-sign-bounds}, $\Dcoef_{\mathrm{KS}}(\widehat x)>6609$. The Lipschitz continuity of $x\mapsto \Dcoef_{\mathrm{KS}}(x)$  implies
$\Dcoef_{\mathrm{KS}}(x^\ast) \geq\Dcoef_{\mathrm{KS}}(\widehat x)-L_D\delta_{\mathrm{KS}}$. Inequality \eqref{eq:KS-D-delta-condition} assures that the right hand side is $>0$.
\end{proof}

The certified bound $L_D\leq 53\,890$, established in Subsection \ref{subsec:certification} on the box $\overline B_\infty(\widehat x,1/32)$, reduces the sufficient condition to $  \delta_{\mathrm{KS}} <\frac{6609}{53\,890}  \approx 0.12264$.\\

Analogous arguments apply to $\Ecoef_{\mathrm{KS}}$. Set
$  r_3:=\frac{\widehat C_3}{2}+\widehat e_3$,
so that $H_3^{\mathrm{KS}}(\rho_3)=r_3\rho_3$. In the same variables,
\begin{align}
  \Ecoef_{\mathrm{KS}}(\rho,\nu,\widetilde\nu)
  =&(\widehat\nu_1-\widehat\nu_-)      \bigl(\widehat H_1^{\mathrm{KS}}-H_-^{\mathrm c}\bigr)     +(s-\widehat\nu_1)      \bigl(\widehat H_2^{\mathrm{KS}}-H_-^{\mathrm c}\bigr)  \notag\\
  &+(\widetilde\nu-s)      \bigl(\widehat H_2^{\mathrm{KS}}-H_+^{\mathrm c}\bigr)     +(\nu-\widetilde\nu)      \bigl(r_3\rho-H_+^{\mathrm c}\bigr),
  \label{eq:KS-variable-entropy-coefficient}
\end{align}
with nonzero derivatives
\begin{align}
  \partial_\rho\Ecoef_{\mathrm{KS}}  =(\nu-\widetilde\nu)\,r_3,\ \ 
  \partial_\nu\Ecoef_{\mathrm{KS}}  =r_3\rho-H_+^{\mathrm c},\ \ 
  \partial_{\widetilde\nu}\Ecoef_{\mathrm{KS}}  =\widehat H_2^{\mathrm{KS}}-r_3\rho. \label{eq:KS-E-rho-derivative}
\end{align}
The following result transfers the certified lower bound $\Ecoef_{\mathrm{KS}}(\hat x)<-1407$ to the exact fan.
\begin{proposition}
\label{prop:KS-conditional-entropy}
Let $U$ contain the segment from $\widehat x$ to $x^\ast$, and suppose $ \sup_{x\in U}|\nabla\Ecoef_{\mathrm{KS}}(x)|\leq L_E$.
If
\begin{equation}\label{eq:KS-E-delta-condition}
  \delta_{\mathrm{KS}}<\frac{1407}{L_E},
\end{equation}
then $\Ecoef_{\mathrm{KS}}(x^\ast)<0$, and, the exact KS fan has a smaller total-energy rate than the classical contact discontinuity.
\end{proposition}

The certified bound $L_E\leq 49\,688$ from Subsection \ref{subsec:certification}, on the same box, gives the sufficient condition $  \delta_{\mathrm{KS}}  <\frac{1407}{49\,688} \approx 0.02832$.

Summarizing Propositions \ref{prop:KS-conditional-action} and \ref{prop:KS-conditional-entropy}, we conclude that
\begin{theorem} \label{thm:KS-conditional-exact-selection}
Assume that
\begin{enumerate}[label=\textup{(\roman*)}]
  \item the correction satisfies the ordering condition
  \eqref{eq:KS-ordering-condition} of
  Lemma \ref{lem:KS-ordering-stability},
  \item $\sup_U|\nabla\Dcoef_{\mathrm{KS}}|\leq L_D$,
  \item $\sup_U|\nabla\Ecoef_{\mathrm{KS}}|\leq L_E$,
  \item
  \begin{equation}\label{eq:KS-combined-delta-condition}
    \delta_{\mathrm{KS}}  <\min\left\{  \frac{6609}{L_D},\, \frac{1407}{L_E}   \right\}.
  \end{equation}
\end{enumerate}
Then, for the exact KS fan,
\begin{equation}
  \Dcoef_{\mathrm{KS}}(x^\ast)>0   \qquad\text{and}\qquad  \Ecoef_{\mathrm{KS}}(x^\ast)<0.   \label{eq:KS-exact-signs-conditional}
\end{equation}
Hence both the action rate and the entropy rate criteria prefer the exact KS convex integration family  to the classical contact
discontinuity.
\end{theorem}

Numerically, we obtain the certified estimate $\delta_{\mathrm{KS}}\leq 2.45\cdot 10^{-13}$, and \eqref{eq:KS-combined-delta-condition} is satisfied with a margin greater than $10^{10}$. In view of Theorem \ref{thm:contact-agreement}, the signs \eqref{eq:KS-exact-signs-conditional} are known in advance. The content of Theorem~\ref{thm:KS-conditional-exact-selection} and of the certified constants of Subsection \ref{subsec:certification} is that the coefficients at the exact fan essentially coincide with the certified rational values.\\

We finally quantify the margin with which the two criteria agree at the exact fan. Suppose
\begin{equation*}
  |\Kcoef_{\mathrm{KS}}(x^\ast)-\Kcoef_{\mathrm{KS}}(\widehat x)|  \leq L_K\delta_{\mathrm{KS}},
  \qquad  |\Icoef_{\mathrm{KS}}(x^\ast)-\Icoef_{\mathrm{KS}}(\widehat x)| \leq L_I\delta_{\mathrm{KS}},
\end{equation*}
and let the agreement margin at the rational vector be defined by
\begin{equation}
  m_{\mathrm{KS}}  :=\frac{\widehat{\Kcoef}_{\mathrm{KS}}}{2} -\bigl|\widehat{\Icoef}_{\mathrm{KS}}\bigr|  =1407.036311937891\ldots.  \label{eq:KS-agreement-margin}
\end{equation}

\begin{proposition}\label{prop:KS-agreement-stability}
If
\begin{equation}\label{eq:KS-agreement-stability-condition}
  \frac{L_K}{2}\,\delta_{\mathrm{KS}} +L_I\,\delta_{\mathrm{KS}}  <m_{\mathrm{KS}}
\end{equation}
and $\Kcoef_{\mathrm{KS}}(x^\ast)>0$, then
\begin{equation} \label{eq:KS-exact-agreement-condition}
  \bigl|\Icoef_{\mathrm{KS}}(x^\ast)\bigr|
  <\frac{\Kcoef_{\mathrm{KS}}(x^\ast)}{2}.
\end{equation}
Hence the two criteria continue to agree at the exact fan.
\end{proposition}

\begin{proof}
We have
\begin{align*}
  \frac{\Kcoef_{\mathrm{KS}}(x^\ast)}{2} -\bigl|\Icoef_{\mathrm{KS}}(x^\ast)\bigr|
  \geq&  \frac{\widehat{\Kcoef}_{\mathrm{KS}}}{2} -\bigl|\widehat{\Icoef}_{\mathrm{KS}}\bigr|
  -\frac12\bigl|    \Kcoef_{\mathrm{KS}}(x^\ast)-\widehat{\Kcoef}_{\mathrm{KS}}
  \bigr|  -\bigl| \Icoef_{\mathrm{KS}}(x^\ast)-\widehat{\Icoef}_{\mathrm{KS}}
  \bigr|.
\end{align*}
The right hand side is positive by \eqref{eq:KS-agreement-stability-condition}.
\end{proof}
Since $\Kcoef=\Dcoef-\Ecoef$ and $\Icoef=(\Dcoef+\Ecoef)/2$ by Theorem \ref{thm:kinetic-internal-decomposition}, one may take $L_K\leq L_D+L_E$ and $L_I\leq(L_D+L_E)/2$. With the certified constants and correction bound of Subsection \ref{subsec:certification}, condition \eqref{eq:KS-agreement-stability-condition} then holds with a margin greater than $10^{10}$.\\

\noindent \textbf{Exact evaluation for the KS parameters.} Appendix A certifies exact values close to those given in \eqref{eq:KS-rational-KI-values}, \eqref{eq:KS-rational-values}. In Theorem \ref{thm:unconditional} we show:
$$\Dcoef_{\mathrm{KS}}(x^*)\;>\;6609-2\cdot 10^{-8}\;>\;0,\qquad \Ecoef_{\mathrm{KS}}(x^*)\;<\;-1407+2\cdot 10^{-8}\;<\;0 .$$
Also, $\Kcoef_{\mathrm{KS}}(x^*) \simeq 8016.7142$  and $\Icoef_{\mathrm{KS}}(x^*) \simeq 2601.3208$. Both pairwise criteria prefer the exact Krupa-Sz\'{e}kelyhidi convex integration family to the classical contact discontinuity, with the certified margins above.

\section{Horimoto's contact discontinuity}\label{sec:Horimoto}

Horimoto \cite{KrupaSzekelyhidi} has established non-uniqueness for symmetric data corresponding to a contact discontinuity for every pressure law  $0<  p\in C^1(0,\infty)$, with $p'(\rho)>0$ for all $\rho>0$, using the refined convex-integration framework of Markfelder
\cite[Theorem~1.3]{Horimoto2026}.  In particular, unlike the construction from \cite{KrupaSzekelyhidi}, the existence of the relaxed fan does not rely on a specially designed pressure law.  We show that, at the relaxed level, both the action rate and entropy rate citeria prefer Horimoto's fan for every such pressure law.

Let us write $P(\rho):=\rho\eint(\rho)$ for the pressure potential. Then
$P(\rho)  = \rho\int_{\rho_\ast}^{\rho} \frac{p(r)}{r^2}\,dr$,
and $P$ is strictly convex,
\begin{equation}
  P''(\rho)=\frac{p'(\rho)}{\rho}>0.
  \label{eq:Horimoto-strict-convexity}
\end{equation}
Horimoto considers the stationary contact data
\begin{equation*}
  \rho_-=\rho_+=\rho_0,
  \qquad
  m_-=  \begin{pmatrix}    -\rho_0u_0\\0 \end{pmatrix},
  \qquad  m_+=  \begin{pmatrix}   \rho_0u_0\\0  \end{pmatrix},
\end{equation*}
where $\rho_0>0$ and $u_0\neq0$.  The  self-similar reference solution is the stationary contact $y=0$. The relaxed fan has interface speeds $-b<-a<a<b$,  $0<a<b$, and symmetric intermediate densities $ \rho_1=\rho_3  =  \rho_0+\frac{a\varepsilon}{b-a}$, $\rho_2=\rho_0-\varepsilon$,
where $0<\varepsilon<\rho_0$.  In particular,
\begin{equation}
  (b-a)\rho_1+a\rho_2=b\rho_0.
  \label{eq:Horimoto-mass-identity}
\end{equation}

In Horimoto's relaxed formulation, $q_i-p(\rho_i)$ plays the role of the kinetic energy density.  Thus, in the notation
of Sections \ref{sec:finite-dimensional-formulas} and \ref{sec:comparisonstability}, $\frac12\rho_iC_i$ corresponds to $q_i-p(\rho_i)$. Under this identification, the formulas for entropy and action rate coefficients and the kinetic/internal energy decomposition apply verbatim.
We hence define 
\begin{align}
  \overline H_i  :=   q_i+P(\rho_i)-p(\rho_i),  \qquad   \overline\Lambda_i   := q_i-p(\rho_i)-P(\rho_i) \qquad (i=1,2).  \label{eq:Horimoto-relaxed-Lagrangian}
\end{align}
By symmetry the third state has the same densities as the first state.

Since for the classical contact discontinuity, 
$$  H_0   := \frac12\rho_0u_0^2+P(\rho_0),  \qquad
  \Lambda_0  :=   \frac12\rho_0u_0^2-P(\rho_0), $$
the relative entropy rate and  action rate coefficients are given by
\begin{equation*}
  \overline\Ecoef_{\mathrm H}  :=   2(b-a)(\overline H_1-H_0)  +  2a(\overline H_2-H_0),\quad
  \overline\Dcoef_{\mathrm H}  := 2(b-a)(\Lambda_0-\overline\Lambda_1)  +  2a(\Lambda_0-\overline\Lambda_2).
\end{equation*}
The corresponding kinetic and internal energy mismatches are
\begin{align*}
  \overline\Kcoef_{\mathrm H}  &:=  2(b-a)  \left[   \rho_0u_0^2 -    2\bigl(q_1-p(\rho_1)\bigr)  \right]+  2a  \left[  \rho_0u_0^2   -    2\bigl(q_2 -p(\rho_2)\bigr)\right],
\\ \Icoef_{\mathrm H} &:= 2(b-a)  \bigl( P(\rho_1)-P(\rho_0)  \bigr) +2a\bigl(  P(\rho_2)-P(\rho_0)  \bigr).
\end{align*}
Use of Theorem \ref{thm:kinetic-internal-decomposition},
\begin{equation} \label{eq:Horimoto-relaxed-KI}
  \overline\Dcoef_{\mathrm H}   =  \frac12\overline\Kcoef_{\mathrm H}  +  \Icoef_{\mathrm H},  \qquad   \overline\Ecoef_{\mathrm H}  =  -\frac12\overline\Kcoef_{\mathrm H} +  \Icoef_{\mathrm H}.
\end{equation}
We now show the location of the relaxed Horimoto fan.
\begin{proposition} \label{prop:Horimoto-relaxed-selection}
For the relaxed fan constructed by Horimoto,
\begin{equation}
  \overline\Kcoef_{\mathrm H}>0,  \qquad  0<\Icoef_{\mathrm H}  <  \frac{\overline\Kcoef_{\mathrm H}}{2}.
  \label{eq:Horimoto-agreement-sector}
\end{equation}
Consequently,  $\overline\Dcoef_{\mathrm H}>0$ and $\overline\Ecoef_{\mathrm H}<0$.
\end{proposition}
\begin{proof}
Recall from \eqref{eq:Horimoto-strict-convexity} that the pressure potential $P$ is strictly convex.  Equation
\eqref{eq:Horimoto-mass-identity} and strict Jensen convexity then give $(b-a)P(\rho_1)+aP(\rho_2) >  bP(\rho_0)$,
because $\rho_1$, $\rho_2$ are not both equal to $\rho_0$. Hence $\Icoef_{\mathrm H}>0$.

From the reduced entropy condition (3.10) in \cite[Lemma~3.3]{Horimoto2026}, we deduce that
\begin{equation*}
  b(\overline H_1-H_0)   \leq  a(\overline H_1-\overline H_2).
\end{equation*}
The proof of \cite[Lemma~3.4]{Horimoto2026} shows that this inequality is strict, so that the entropy gap
\begin{equation*}
  \gamma_{\mathrm H} :=  a(\overline H_1-\overline H_2) -  b(\overline H_1-H_0)  >0.
\end{equation*}
We compute 
\begin{equation}  \label{eq:Horimoto-E-gap}
  \overline\Ecoef_{\mathrm H}  =  -2\gamma_{\mathrm H}<0.
\end{equation}
Expressions \eqref{eq:Horimoto-relaxed-KI} and \eqref{eq:Horimoto-E-gap} combine to yield $  \overline\Kcoef_{\mathrm H}  =  2\Icoef_{\mathrm H}  +  4\gamma_{\mathrm H}>0$ and $\overline\Dcoef_{\mathrm H} =  2\Icoef_{\mathrm H}  +  2\gamma_{\mathrm H}>0$.
In particular, $  \frac{\overline\Kcoef_{\mathrm H}}2 -  \Icoef_{\mathrm H} =  2\gamma_{\mathrm H}>0$,
and \eqref{eq:Horimoto-agreement-sector} follows.
\end{proof}
From \eqref{eq:Horimoto-agreement-sector} a relaxed Horimoto fan therefore lies in the upper part of the sector in Figure \ref{fig:KI-phase-diagram} where entropy rate and action rate criteria both prefer the fan. No specially designed equation of state or computer-assisted certification is required.

For illustration, the point $H$ in  Figure \ref{fig:KI-phase-diagram}  marks the relaxed Horimoto fan for a representative parameter choice ($\gamma=2$, $\rho_0=u_0=1$, $\varepsilon=1/16$, $a=1/8$). The arrow indicates the range of the sector spanned by further members of this family, which approach the $\Kcoef$-axis as $\varepsilon\to 0$ and the boundary $\Icoef=\bar{\Kcoef}/2$ as the admissibility gap $\gamma_{\mathrm H}\to 0$. \\

\noindent \textbf{Individual exact realizations.} 
The previous Proposition concerns the piecewise constant relaxed fan. The exact solutions obtained from Horimoto's fan differ from the solutions discussed above, like \cite{KrupaSzekelyhidi}, in that Markfelder's realization theorem does not prescribe one fixed value of the kinetic energy density $\frac{|m|^2}{2\rho}$ in each turbulent region.  The constant fan coefficients $\Ecoef$, $\Dcoef$ of Section \ref{sec:finite-dimensional-formulas} are therefore not directly attached to an individual Horimoto realization, and the comparison for exact realizations therefore requires the additional one-sided energy comparison discussed below.
\begin{assumption} \label{ass:Horimoto-normalization}
Let $(\rho_{\mathrm H},m)$ be a particular exact realization from \cite{Horimoto2026}.  We assume the tangentially periodic normalization used in
Remark \ref{rem:renormal}: either $m$ is periodic in the tangential variable with period $\ell$, or the corresponding per-unit-length spatial means exist and the cutoff errors in the energy flux vanish.
\end{assumption}
Set
\begin{equation*}
  H_m  :=  \frac{|m|^2}{2\rho_{\mathrm H}}  +  P(\rho_{\mathrm H}),  \qquad   \Lambda_m  :=  \frac{|m|^2}{2\rho_{\mathrm H}} -  P(\rho_{\mathrm H}).
\end{equation*}
Let $\overline H$ and $\overline\Lambda$ denote the piecewise
constant relaxed densities defined by
\eqref{eq:Horimoto-relaxed-Lagrangian}.
Because the exact and relaxed solutions have the same density,
\begin{equation} \label{eq:Horimoto-action-energy-difference}
  \Lambda_m-\overline\Lambda =   H_m-\overline H.
\end{equation}
We also define the spatially averaged difference
  $G_m(t)   := \int_{\mathbb T_\ell\times\mathbb R}
  \bigl(    H_m-\overline H \bigr)(t,x,y)
  \,dx\,dy$. For an exact realization we obtain that the entropy and action rate-criteria prefer the exact realization to the classical contact:
\begin{proposition} \label{prop:Horimoto-exact-selection}
Under Assumption \ref{ass:Horimoto-normalization}, $G_m$ has a non-increasing representative and
\begin{equation} \label{eq:Horimoto-Gm-negative}
  G_m(t)\leq0  \qquad  \text{for a.e. }t>0.
\end{equation}
Consequently, at every Lebesgue time $t>0$, $\Energy_m(t)-\Energy_{\mathrm c}(t) \leq   \ell t\,   \overline\Ecoef_{\mathrm H} <0$.
and for every $t>0$, $\A_{\mathrm c}(t)-\A_m(t)   \geq  \frac{\ell t^2}{2}   \overline\Dcoef_{\mathrm H} >0$.
\end{proposition}
\begin{proof}
Proposition~2.5 in \cite{Horimoto2026} records the existence of convex integration solutions from Markfelder's realization theorem.  For the additional comparison we use  \cite[Proposition 4.5, equation (4.9)]{Markfelder2024}, which gives
\begin{align}
  &\int   \left[H_m\,\partial_t\psi +    \left(     H_m+p(\rho_{\mathrm H})    \right)    \frac{m}{\rho_{\mathrm H}}
    \cdot\nabla\psi  \right]  \,dx\,dy\,dt  \geq  \int  \left[    \overline H\,\partial_t\psi    +    \overline F\cdot\nabla\psi \right]  \,dx\,dy\,dt  \label{eq:Horimoto-Markfelder-comparison}
\end{align}
for every non-negative test function $\psi$.

Using tangential periodicity or per-unit-length cutoffs, choose spatial cutoffs that are equal to one on the expanding fan and pass to the limit.  This readily shows that $\partial_tG_m\leq0$ in distributions.

The exact and relaxed fields agree outside the expanding fan, whose cross-sectional measure per tangential period is of the order of $t$.  Since the
fields are bounded, $ G_m(t)\rightarrow0$ as $t \to0^+$.
The non-increasing representative therefore satisfies \eqref{eq:Horimoto-Gm-negative}. Now, on using \eqref{eq:Horimoto-action-energy-difference},
$$  \A_m(T)-\overline\A_{\mathrm H}(T) = \int_0^T G_m(t)\,dt  \leq0,$$
which in combination with $  \A_{\mathrm c}(T)-\overline\A_{\mathrm H}(T)=\frac{\ell T^2}{2}  \overline\Dcoef_{\mathrm H}$ proves the required positivity of the action difference.

Similarly, we find $$  \Energy_m(t)-\Energy_{\mathrm c}(t) = G_m(t) + \ell t\overline\Ecoef_{\mathrm H},$$ and \eqref{eq:Horimoto-Gm-negative} then shows the negativity of the energy difference.
\end{proof}
Without Assumption \ref{ass:Horimoto-normalization}, Horimoto's whole-space existence theorem does not by itself define a scalar global
action or total-energy rate.  The local distributional comparison \eqref{eq:Horimoto-Markfelder-comparison} remains valid, but converting
it into the scalar inequalities above requires the assumption.

We finally interpret the effective position of a realization in Figure \ref{fig:KI-phase-diagram}.
For almost every $t>0$, we may define
\begin{equation*}
  \Kcoef_m(t)  :=  \overline\Kcoef_{\mathrm H} -  \frac{2G_m(t)}{\ell t}.
\end{equation*}
Moreover, $\Icoef_m(t)=\Icoef_{\mathrm H}$ is independent of time. Since $G_m(t)\leq0$, we find  $\Kcoef_m(t)  \geq  \overline\Kcoef_{\mathrm H} >  2\Icoef_{\mathrm H}>0$.
The point $\bigl( \Kcoef_m(t),\Icoef_{\mathrm H}  \bigr)$ then lies on the horizontal ray $\Icoef=\Icoef_{\mathrm H}$
to the right of the relaxed Horimoto point. $ \bigl(    \Kcoef_m(t),\Icoef_{\mathrm H}  \bigr)$ provides the exact instantaneous comparisons:
\begin{equation*}
  \Dcoef_m^{\mathrm{eff}}(t) :=   \frac{1}{\ell t} \frac{d}{dt}  \left(    \A_{\mathrm c}-\A_m  \right)(t)  =  \frac12\Kcoef_m(t)  + \Icoef_{\mathrm H},
  \ \
  \Ecoef_m^{\mathrm{eff}}(t) :=  \frac{  \Energy_m(t)-\Energy_{\mathrm c}(t)
  }{\ell t}  = -\frac12\Kcoef_m(t)  +  \Icoef_{\mathrm H}.
\end{equation*}
It remains strictly inside the sector in which both criteria prefer the wild solution.

Unlike for the fan solutions considered above, however, $\Kcoef_m$ may depend on the realization and on time.

\appendix

\section{Appendix: Exact symbolic verification}\label{sec:symbolic-verification}

The computational statements of this paper are verified by two independent ancillary programs. The Maple worksheet
\texttt{KS\_coefficients.mpl} evaluates the four coefficients $\widehat{\Dcoef}_{\mathrm{KS}}$, $\widehat{\Ecoef}_{\mathrm{KS}}$,
$\widehat{\Kcoef}_{\mathrm{KS}}$, $\widehat{\Icoef}_{\mathrm{KS}}$ at the rational parameter vector, and the Python script
\texttt{certify.py} certifies the correction bound and the Lipschitz constants required for the passage to the exact fan
(Subsection \ref{subsec:certification}). In both programs all quantities are represented as exact rational numbers, and
floating-point output is produced only after every relevant inequality and identity has been verified symbolically.\\

\noindent\textbf{Maple worksheet.} The worksheet \texttt{KS\_coefficients.mpl}\footnote{Available at \url{https://mat1.uibk.ac.at/heiko/KS_coefficients.mpl}.} performs the following steps, aborting with an error if any check fails:

\begin{enumerate}[label=\textup{(\roman*)}]   
\item it loads the exact rational data of  Table \ref{tab:KS-rational-data} and verifies the wave ordering   $\widehat\nu_-<\widehat\nu_1<\widehat s<\widehat\nu_2<\widehat\nu_+$,  which underlies the overlap formulas of  Proposition \ref{prop:KS-action-coefficient},
  \item it forms the Lagrangian and mechanical energy densities   $\Lambda_\pm^{\mathrm c}$, $H_\pm^{\mathrm c}$,  $\widehat\Lambda_i^{\mathrm{KS}}$, $\widehat H_i^{\mathrm{KS}}$ and  evaluates $\widehat{\Dcoef}_{\mathrm{KS}}$,  $\widehat{\Ecoef}_{\mathrm{KS}}$ from  \eqref{eq:three-region-action-coefficient},  \eqref{eq:three-region-entropy-coefficient}, verifying the strict
  sign bounds \eqref{eq:KS-rational-sign-bounds} by symbolic checks,
  \item it evaluates $\widehat{\Kcoef}_{\mathrm{KS}}$,   $\widehat{\Icoef}_{\mathrm{KS}}$ both from their definitions and
  from the identities $\Kcoef=\Dcoef-\Ecoef$,   $\Icoef=(\Dcoef+\Ecoef)/2$, and verifies exactly that the two
  evaluations coincide and that the decomposition  $\Dcoef=\tfrac12\Kcoef+\Icoef$, $\Ecoef=-\tfrac12\Kcoef+\Icoef$ of  Theorem \ref{thm:kinetic-internal-decomposition} holds,
  \item it verifies the agreement inequalities  \eqref{eq:KS-rational-agreement} by guarded symbolic checks,  together with the identity  $\widehat{\Kcoef}_{\mathrm{KS}}/2  -|\widehat{\Icoef}_{\mathrm{KS}}| =-\widehat{\Ecoef}_{\mathrm{KS}}$, valid because
  $\widehat{\Icoef}_{\mathrm{KS}}>0$,
  \item it then prints the twenty-digit decimal expansions reported in \eqref{eq:KS-rational-values}, \eqref{eq:KS-rational-KI-values} and in  \eqref{eq:KS-agreement-margin}.
\end{enumerate}
An optional diagnostic section evaluates the four regional contributions to each coefficient separately.\\

\noindent \textbf{Passage to the exact fan.} The worksheet proves the strict inequalities \eqref{eq:KS-rational-sign-bounds} unconditionally at the rational parameter vector. To transfer these signs to the exact corrected fan, requires supplementation by the certified bounds
\begin{equation*}
  \delta_{\mathrm{KS}}\leq\delta_0, \qquad  \sup_U|\nabla\Dcoef_{\mathrm{KS}}|\leq L_D,  \qquad  \sup_U|\nabla\Ecoef_{\mathrm{KS}}|\leq L_E.
\end{equation*}
By Theorem \ref{thm:KS-conditional-exact-selection}, the exact-fan inequalities then follow whenever
\begin{equation}  \label{eq:certified-action-condition}
  L_D\,\delta_0<6609
\end{equation}
and
\begin{equation}  \label{eq:certified-entropy-condition}
  L_E\,\delta_0<1407.
\end{equation}

The symbolic calculations (using the Maple worksheet) are statements about the rational parameter vector. Corresponding statements for the exact fan are conditional only on the certified correction and gradient estimates specified in \eqref{eq:certified-action-condition},
\eqref{eq:certified-entropy-condition}. Both conditions are established next.

\subsection{Certification of the correction size and  verification of \eqref{eq:certified-action-condition}, \eqref{eq:certified-entropy-condition}}
\label{subsec:certification}

In this subsection we verify the conditions \eqref{eq:certified-action-condition}, \eqref{eq:certified-entropy-condition} directly from
the computational material accompanying the Krupa-Sz\'ekelyhidi construction \cite{KrupaSzekelyhidi,KScode}. This removes the conditional
character of Theorem \ref{thm:KS-conditional-exact-selection} for the explicit KS fan. All quantities below are computed in exact rational arithmetic. Square roots are bounded by integer-square-root enclosures and positive definiteness is certified by exact $LDL^{\top}$ factorizations. The complete verification is performed by the ancillary script \texttt{certify.py}, which depends only on the Python standard library and reruns from scratch in a few seconds.\footnote{Available at \url{https://mat1.uibk.ac.at/heiko/certify.py}. The script also verifies in exact arithmetic the symbolic checks of \cite{KScode} at the rational vector, including the identity of the input data of Table \ref{tab:KS-rational-data} with the appendix constants  of \cite{KrupaSzekelyhidi}, the exact validity of the Rankine--Hugoniot conditions at the interfaces $\nu_-$ and $\nu_1$, the contact identifications, and all $27$ strict inequality constraints not involving the corrected variables.}\\

In the parametrization of \cite{KrupaSzekelyhidi}, each turbulent region carries the relaxed velocity $v_i=(\alpha_i,\beta_i)$, the trace-free stress $u_i$ encoded by $(\gamma_i,\delta_i)$, the kinetic level $C_i$, and the values $e_i$, $e_i'$ of the designed internal energy function at $\rho_i$. Inspection of the equality constraints shows that the six corrected variables
$$x=(\alpha_3,\beta_3,\delta_3,\rho_3,\nu_+,\nu_2) $$
of \eqref{eq:KS-parameter-vector} enter only the Rankine-Hugoniot conditions at the two right-most interfaces $\nu_2$ and $\nu_+$. The remaining Rankine-Hugoniot conditions at $\nu_-$ and $\nu_1$ and the contact identifications hold \emph{exactly} at the rational vector $\widehat x$ and do not involve $x$. Let $\Gamma:\mathbb{R}^6\to\mathbb{R}^6$ denote the vector of these six residuals, fixing all remaining parameters at the above rational values and with the normalization $e''(\widehat\rho_3)=0$ adopted in \cite{KrupaSzekelyhidi,KScode}, so that $e'(\rho_3)\equiv\widehat e_3'$ near $\widehat\rho_3$. A zero of $\Gamma$ corresponds to satisfying all equality constraints exactly, and the exact vector $x^*$ of parameters from \cite{KrupaSzekelyhidi} is a zero of $\Gamma$.

We summarize the results of the exact computations:
\begin{proposition} \label{prop:certdata}
The following statements hold in exact rational arithmetic.
\begin{itemize}
\item[(i)] The rational inputs of Table \ref{tab:KS-rational-data} coincide exactly with the constants from the appendix of  \cite{KrupaSzekelyhidi}. The Rankine--Hugoniot conditions at $\nu_-$ and $\nu_1$ hold exactly at $\widehat x$.
\item[(ii)] Each component of $\Gamma(\widehat x)$ is smaller than $10^{-11}$ in absolute value, and
$\|\Gamma(\widehat x)\|_2\le 3.18\cdot 10^{-12}$.
\item[(iii)] The Jacobian $D\Gamma(\widehat x)$, computed exactly, coincides (up to a global sign) with the matrix \texttt{DGamma} of \cite{KScode}, and
$$
\sigma_{\min}\bigl(D\Gamma(\widehat x)\bigr)\;\ge\;2.155 .
$$
\item[(iv)] The exact Newton step $\eta:=\bigl\|D\Gamma(\widehat x)^{-1}\Gamma(\widehat x)\bigr\|_2$
satisfies $\eta\le 1.225\cdot 10^{-13}$.
\item[(v)] On the box $\overline B_\infty(\widehat x,2^{-20})$ the Jacobian is Lipschitz with constant $K\le 52$ with respect to the Euclidean norm.
\end{itemize}
\end{proposition}
These estimates imply the following for the exact solution: 
\begin{theorem}
\label{thm:deltacert}
Set $\delta_0:=2\eta\le 2.45\cdot 10^{-13}$. Then $\Gamma$ has a zero $x^*$ with
$$ |x^*-\widehat x|\;\le\;\delta_0,$$
and this zero is unique in the ball of radius $0.08$ around $\widehat x$. In particular, it is the exact KS parameter vector, and
$$\delta_{KS}\;\le\;\delta_0\;\le\;2.45\cdot 10^{-13}.$$
Moreover, the wave ordering $\nu_1<s<\nu_2^*<\nu_+^*$ of Lemma \ref{lem:KS-ordering-stability}  is
preserved, and every strict constraint of the fan subsolution involving the corrected variables (the two subsolution conditions in the third region, the admissibility inequalities at $\nu_2$ and $\nu_+$, the speed ordering, the convexity conditions involving $(\rho_3,e_3)$, and positivity of $\rho_3$) persists at $x^*$: each holds at $\widehat x$ with slack at least $0.759$ beyond the strictness margin $\varepsilon=\tfrac13$ of \cite{KScode}, while its certified variation along the correction is at most $1.4\cdot 10^{-8}$. Consequently $x^*$ is a strict admissible fan subsolution.
\end{theorem}
\begin{proof}
By Proposition \ref{prop:certdata}, $\beta:=\|D\Gamma(\widehat x)^{-1}\|_2\le \sigma_{\min}^{-1}\le 2.155^{-1}$ and the Newton--Kantorovich constant satisfies
$$h\;=\;\beta K\eta\;\le\;2.96\cdot 10^{-12}\;\le\;\tfrac12 .$$
The Newton-Kantorovich theorem \cite{OR70} yields a zero $x^*$ in $\overline B(\widehat x,r_-)$ with $r_-=(1-\sqrt{1-2h})/(\beta K)\le 2\eta=\delta_0\le 2^{-20}$, so the Lipschitz bound (v) applies along the whole segment, together with uniqueness in $\overline B(\widehat x,r_+)$, $r_+=(1+\sqrt{1-2h})/(\beta K)\ge 0.08$. Since $\delta_0<\min\{\widehat\nu_2-s,\ (\widehat\nu_+-\widehat\nu_2)/2\} =5.529\ldots$, Lemma \ref{lem:KS-ordering-stability} applies. The persistence of the strict inequalities follows from the Lipschitz continuity of each constraint function. The certified suprema of their gradients on $\overline B_\infty(\widehat x,2^{-20})$ do not exceed $5.48\cdot 10^{4}$, whence the variation is at most $5.48\cdot 10^{4}\cdot\delta_0\le 1.4\cdot 10^{-8}$, far below the minimal slack $0.759$.
\end{proof}

We finally verify the action and entropy rate criteria \eqref{eq:certified-action-condition} and \eqref{eq:certified-entropy-condition}:
\begin{theorem}
\label{thm:unconditional}
On the box $\overline B_\infty(\widehat x,1/32)$, which contains the segment from $\widehat x$ to $x^*$, the coefficients admit the Lipschitz bounds
$$L_D\le 53\,890,\qquad L_E\le 49\,688 .$$
Together with Theorem \ref{thm:deltacert},
$$L_D\,\delta_0\;\le\;1.4\cdot 10^{-8}\;<\;6609, \qquad L_E\,\delta_0\;\le\;1.3\cdot 10^{-8}\;<\;1407,$$
so \eqref{eq:certified-action-condition} and \eqref{eq:certified-entropy-condition} hold. Consequently, the hypotheses of Theorem \ref{thm:KS-conditional-exact-selection} are satisfied and
$$\Dcoef_{\mathrm{KS}}(x^*)\;>\;6609-2\cdot 10^{-8}\;>\;0,\qquad \Ecoef_{\mathrm{KS}}(x^*)\;<\;-1407+2\cdot 10^{-8}\;<\;0 .$$
\end{theorem}
Both the pairwise entropy and action rate criteria prefer the exact Krupa-Sz\'ekelyhidi convex integration family to the classical contact discontinuity.

\begin{remark} \label{rem:sixseven}
The thermodynamic normalization of \cite{KrupaSzekelyhidi,KScode} fixes $C_3^*=\widehat C_3$ and $e''(\widehat\rho_3)=0$, so that the designed internal energy is affine near $\widehat\rho_3$ and 
$$e(\rho_3^*)=\widehat e_3+\widehat e_3'\,(\rho_3^*-\widehat\rho_3),
\qquad \bigl|e(\rho_3^*)-\widehat e_3\bigr| \le \widehat e_3'\,\delta_0\le1.2\cdot 10^{-12},$$
rather than $e(\rho_3^*)=\widehat e_3$ exactly. Repeating the computation of Subsection \ref{sec:KS-geometry} with the corresponding chain rule which replaces the slopes $q_3$ and $r_3$ of $\Lambda_3^{\mathrm{KS}}(\rho_3)$ and $H_3^{\mathrm{KS}}(\rho_3)$ in \eqref{eq:KS-D-rho-derivative} and \eqref{eq:KS-E-rho-derivative} by $q_3-\rho\,\widehat e_3'$ and $r_3+\rho\,\widehat e_3'$ yields the certified constants
$L_D\le 52\,834$ and $L_E\le 48\,644$ on the same box, and the inequalities of Theorem \ref{thm:unconditional} hold verbatim.
\end{remark}

The script \texttt{certify.py} verifies, in order: the input identity of (i), the exact residuals of (ii), the entrywise agreement of the exact Jacobian with \texttt{DGamma} of \cite{KScode}, the lower bound (iii) by bisection with exact $LDL^{\top}$ certificates on $D\Gamma^{\top}D\Gamma-\lambda I$, the Newton step (iv) by exact Gaussian elimination. the Lipschitz bounds (v) and those of Theorem \ref{thm:unconditional} by evaluating the affine gradients of all quadratic quantities at the vertices of the corresponding boxes, and finally the products $L_D\delta_0$, $L_E\delta_0$, the ordering condition of Lemma \ref{lem:KS-ordering-stability}, and the persistence estimates of Theorem \ref{thm:deltacert}. No floating-point operation enters any verified inequality.

%\vspace*{-0.2cm}
%
%\section*{Data Availability Statement} Accompanying source codes are publically available. No datasets were generated or analysed during the current study.
%
%\vspace*{-0.2cm}
%
%\section*{Conflicts of Interest} The authors have no financial or proprietary interests in any material discussed in this article.
%
%\vspace*{-0.3cm}


\begin{thebibliography}{99}

\vspace*{-0.2cm}

\bibitem{BFH} D.~Breit, E.~Feireisl, M.~Hofmanov\'{a}, Solution Semiflow to the Isentropic Euler System, Archive for Rational Mechanics and Analysis \textbf{235} (2020), 167--194.

\vspace*{-0.2cm}


\bibitem{CFKM} E. Chiodaroli, E. Feireisl, O. Kreml, and S. Markfelder, Maximal entropy production principle and the Euler system of gas dynamics, arXiv:2605.26687 (2026).

\vspace*{-0.2cm}

\bibitem{ChiodaroliDeLellisKreml} E. Chiodaroli, C. De Lellis, and O. Kreml, Global ill‐posedness of the isentropic system of gas dynamics, Comm. Pure Appl. Math {\bf 68} (2015), 1157--1190.

\vspace*{-0.2cm}

\bibitem{ChiodaroliKreml} E. Chiodaroli and O. Kreml, On the energy dissipation rate of solutions of the compressible isentropic Euler system,  Arch. Rat. Mech. Anal. {\bf 214} (2014), 1019--1049.     

\vspace*{-0.2cm}

\bibitem{DafermosEntropyRate} C. M. Dafermos, The entropy rate admissibility criterion for solutions of hyperbolic conservation laws,  J. Differ. Equ. {\bf 14} (1973), 202--212.

\vspace*{-0.2cm}

\bibitem{Dafermos26} C. M. Dafermos,  Shock splitting and entropy production, Discrete and Continuous Dynamical Systems \textbf{55} (2026), 484--496.

\vspace*{-0.2cm}

\bibitem{DeLellisSzekelyhidi} C. De Lellis and L. Sz\'ekelyhidi, The Euler equations as a differential inclusion, Ann. Math. {\bf 170} (2009), 1417--1436.

\vspace*{-0.2cm}

\bibitem{Feireisl} E. Feireisl,  Maximal dissipation and  well-posedness for the compressible Euler system, J. Math. Fluid  Mech. {\bf 16} (2014), 447--461.   

\vspace*{-0.2cm}

\bibitem{GK} B. Gebhard and J. J. Kolumb\'{a}n, The Rayleigh--Taylor instability with local energy dissipation, Mathematische Annalen \textbf{393} (2025), 3283--3336.

\vspace*{-0.2cm}

\bibitem{GKKS1} H. Gimperlein, M. Grinfeld, R. J. Knops, and M. Slemrod, The least action admissibility principle, Arch. Rat. Mech. Anal. \textbf{249} (2025), 22.

\vspace*{-0.2cm}

\bibitem{GKKS2} H. Gimperlein, M. Grinfeld, R. J. Knops, and M. Slemrod, On action rate admissibility criteria, Zeitschrift für Angewandte Mathematik und Physik \textbf{77} (2026), 57.

\vspace*{-0.2cm}

\bibitem{GSGW} P.~ Gwiazda, A.~Swierczewska-Gwiazda, E.~Wiedemann, Weak-strong uniqueness for measure-valued solutions of some compressible fluid models. Nonlinearity \textbf{28} (2015), 3873--3890.

\vspace*{-0.2cm}

\bibitem{Horimoto2026} K.~Horimoto, Non-uniqueness of admissible weak solutions to the two-dimensional barotropic compressible Euler system with contact discontinuities, arXiv:2603.23921.

\vspace*{-0.2cm}

\bibitem{KrupaSzekelyhidi}S. G. Krupa and L. Sz\'ekelyhidi Jr., Contact Discontinuities for 2-D Isentropic Euler are Unique in 1-D but Wildly Non-unique Otherwise, Communications in Mathematical Physics {\bf 406}  (2025), 109.

\vspace*{-0.2cm}

 \bibitem{KScode} S. G. Krupa and L. Sz\'ekelyhidi Jr., MATLAB code associated with \cite{KrupaSzekelyhidi},
 \url{https://github.com/sammykrupa/NonUniqueness2DIsentropicEuler}.

\vspace*{-0.2cm}

\bibitem{Markfelder2024} S.~Markfelder, A new convex integration approach for the compressible Euler equations and failure of the local maximal dissipation criterion,  \emph{Nonlinearity} \textbf{37} (2024), 115022.

\vspace*{-0.2cm}


\bibitem{m21} S. Markfelder, Convex Integration Applied to the Multi-Dimensional Compressible Euler Equations, Lecture Notes in Mathematics 2294, Springer Nature, Cham, 2021.

\vspace*{-0.2cm}

\bibitem{MP1} S. Markfelder and V. Pellhammer, Failure of the least action admissibility principle in the context of the compressible Euler equations, SIAM Journal on Mathematical Analysis \textbf{58} (2026), 3585--3598.

\vspace*{-0.2cm}

\bibitem{MP2} S. Markfelder and V. Pellhammer, The local least action criterion fails as a selection criterion for weak solutions of the compressible Euler equations, arXiv:2606.16685.

\vspace*{-0.2cm}

 \bibitem{OR70} J. M. Ortega and W. C. Rheinboldt, Iterative Solution of Nonlinear Equations in Several Variables, Academic Press, New York, 1970.
\end{thebibliography}
\end{document}